\documentclass[onefignum,onetabnum]{siamart190516}

\usepackage{tikz}
\usepackage{mathrsfs}
\usepackage{enumerate}
\usepackage{framed}
\usepackage{subfig}
\usepackage{algorithmic}

\usepackage[normalem]{ulem}

\newcommand{\R}{\mathbb{R}}

\newsiamremark{Remark}{Remark} 
\crefname{Remark}{Remark}{Remark} 

\usepackage{lipsum}
\usepackage{amsfonts}
\usepackage{graphicx}
\usepackage{epstopdf}
\usepackage{algorithmic}
\ifpdf
  \DeclareGraphicsExtensions{.eps,.pdf,.png,.jpg}
\else
  \DeclareGraphicsExtensions{.eps}
\fi

\newsiamremark{remark}{Remark}
\newsiamremark{hypothesis}{Hypothesis}
\crefname{hypothesis}{Hypothesis}{Hypotheses}
\newsiamthm{claim}{Claim}

\headers{A Modified Split Bregman Algorithm for Nonconvex Energies}{G. Jaramillo and S. C. Venkataramani}

\title{An Example Article\thanks{Submitted to the editors DATE.
\funding{This work was funded by the Fog Research Institute under contract no.~FRI-454.}}}

\author{Dianne Doe\thanks{Imagination Corp., Chicago, IL 
  (\email{ddoe@imag.com}, \url{http://www.imag.com/\string~ddoe/}).}
\and Paul T. Frank\thanks{Department of Applied Mathematics, Fictional University, Boise, ID 
  (\email{ptfrank@fictional.edu}, \email{jesmith@fictional.edu}).}
\and Jane E. Smith\footnotemark[3]}

\usepackage{amsopn}

\ifpdf
\hypersetup{
  pdftitle={A Modified Split Bregman Algorithm for Nonconvex Energies},
  pdfauthor={G. Jaramillo and S. Venkataramani}
}
\fi

\begin{document}
\title{A modified split Bregman algorithm for computing microstructure through Young measures}

\author{Gabriela Jaramillo \thanks{Department of Mathematics, University of Houston, Houston, TX (gabriela@math.uh.edu).}
\and Shankar C. Venkataramani \thanks{Department of Mathematics, University of Arizona, Tucson, AZ (shankar@math.arizona.edu).}}

\maketitle

\begin{abstract}
The goal of this paper is to describe the oscillatory microstructure that can emerge from minimizing sequences for nonconvex energies. We consider integral functionals that are defined on real valued (scalar) functions $u(x)$ which are nonconvex in the gradient $\nabla u$ and possibly also in $u$. To  characterize the microstructures for these nonconvex energies, we minimize the associated relaxed energy using two novel approaches: i) a semi-analytical method based on control systems theory, ii) and a numerical scheme that combines convex splitting together with a modified version of the split Bregman algorithm. These solutions are then used to gain information about minimizing sequences of the original problem and the spatial distribution of microstructure.
\end{abstract}



\section{Introduction}

Macroscopic physical systems consist of large numbers of interacting (microscopic) parts, and are thus described by statistical mechanics \cite{greiner1995thermodynamics}.  A central tenet of statistical mechanics is that the equilibrium state, and the relaxation to equilibrium, are described by an appropriate free energy \cite{greiner1995thermodynamics}. Oftentimes the microscopic degrees of freedom ``self-organize" to spontaneously generate patterns and structures on mesoscopic scales \cite{glansdorff1971thermodynamic,gennes2003simple}. While the details differ, the free energies describing such spontaneous self-organization, a phenomenon also called {\em energy driven pattern formation} \cite{kohn2007energy}, have certain universal features {\em independent of the underlying physical system}. These include (1) nonconvexity of the free energy and the existence of  multiple (usually symmetry related) ground states for the system, and (2) regularization by a singular perturbation (``ultraviolet cutoff") to preclude the formation of structures on arbitrarily fine scales. These features are present in free energies that  describe many systems including liquid crystals \cite{Virga_Variational}, micro-magnetic devices \cite{DKMO:micromagnetics}, non-Euclidean elasticity \cite{efi} and solid-solid phase transitions \cite{kohn1994surface}. 

It is of great interest to develop methods that will lead to an understanding of 
microstructure in a variety of energy-driven systems. As an initial step towards this goal, in this paper, we consider an abstract and much simplified formulation given by the variational problem
\begin{equation}\label{e:original1}
 \min_A I[u] = \min_A \int_\Omega W(\nabla u) + V(x,u) \;dx \hspace{1.5cm} u \in A,
 \end{equation}
where $W(\xi)$ is a nonconvex potential, $V(x,u)$ is continuous in its arguments, and $u(x)$ is a real valued function in an admissible set, which we denote here by $A$. This energy is non-convex and thus has property (1) from above, but it is not regularized, so it does not have property (2). The functional is not, in general, lower semicontinuous, resulting in a lack of classical solutions as possible minimizers. Nonetheless, minimizing sequences for these problems encode useful information \cite{muller1999calculus}. These  minimizing sequences can exhibit finite-amplitude fine-scale oscillations, which in applications correspond to the emergence of microstructures. Indeed, our goal is to characterize spatially heterogeneous microstructures in the context of problems of the form \cref{e:original1}. 

One possible approach to analyze these problems is to consider their regularization via Young measures \cite{Young}. This means that we weaken the formulation through a generalized functional $\tilde{I}$ that depends on parametrized probability measures $\{\nu_x\}_{x \in \Omega}$ rather than on functions $u:\Omega\to\mathbb{R}$. The advantage now is that the Young measure minimizer $\nu_x$ of  $\tilde{I}$ captures the oscillations present in minimizing sequences of the nonconvex functional $I$ near a location $x$. In addition, the generalized functional $\tilde{I}$ is also related to the relaxation of the problem \cref{e:original1}, which is in turn given by  the quasiconvex envelope $\overline{I}$ of the original energy.  The connection between the three problems, the original nonconvex energy, the generalized functional, and the relaxation is given by a theorem by Pedregal \cite{pedregal1995} which states that the minimum of all these energies is the same, and provides a relation between the minimizing Young measure and the solution to the relaxed (quasiconvex) problem.  

The above discussion suggests a possible path for numerically computing microstructures:  Find solutions to the relaxed problem first, and then use Pedregal's theorem to infer the corresponding optimal Young measure. In the one dimensional case this process is straightforward since the quasiconvex envelope of the energy density coincides with its convex envelope. However, although this 1-d problem is easy to set up, the resulting energy density is often nonsmooth and this lack of smoothness is an impediment to computing minimizers.  In this work we present two methods for overcoming this difficulty and thus for finding optimal Young measures for regularized, 1-d, non-convex problems and indicate extensions to multi-dimensional problems.

The first method we present uses a generalized control Hamiltonian together with the Pontryagin Maximum Principle \cite{liberzon2012calculus} to find semi-analytic solutions. In addition, the control Hamiltonian also provides us with a means to check that solutions, found perhaps using a different approach, are indeed minimizers to the relaxed problem. 

Our second approach takes advantage of known algorithms in compressed sensing, where the energies are regularized by adding the (nonsmooth) $L^1$ norm. In particular, we use the split Bregman algorithm \cite{osher2005,goldstein2009} which is easy to code and provides fast convergence (see \cite{Bregman1966relaxation} for the initial formulation of the Bregman method to determine the joint feasibility of a collection of convex constraints,  and \cite{yang2010, goldstein2010, schaeffer6634, tran2015}
 for other applications of the split Bregman method). As in the original algorithm, our modified scheme also decouples the variable $u$ and its gradient $u_x$ via a constraint, allowing us to carry out the minimization in two steps. In the first step we use Gauss-Seidel to solve for the minimizers of the smooth component of our functional, while in the second step we use a proximal operator \cite{combettes2011proximal,parikh2013proximal} to minimize the non smooth component. In addition our scheme sets up the minimization problem through the associated gradient flow. This improves the stability properties of the variational equation associated with the smooth component of the energy functional, and also allows us to use convexity splitting in the case of problems with a nonconvex potential $V(x,u)$. 
 
We  note that while our numerical approach is novel, the idea of numerically minimizing the relaxed energy to find the optimal Young measure (and thus allowing us to understand microstructures) is not new. For energies defined over scalar valued functions, this concept was already exploited in the work of Nicolaides and Walkington \cite{nicolaides1993},  and expanded by Pedregal \cite{pedregal1995}.  In particular, Pedregal proved a relaxation theorem for the corresponding discretized problem, thus establishing a connection between the numerical solution of the relaxation and the optimal discretized measure. Moreover, he showed that for one dimensional problems with nonconvex potentials of the form used here, i.e. $W(\xi) = (\xi^2-1)^2$, the sequence of discretized Young measures converges to the true optimal measure if and only if the corresponding sequence of minimizer of the discretized relaxation converge strongly to the true solution \cite{pedregal1995}. 
 
The above results were later generalized to the case of vector valued functions by Roub{\'i}{\v c}ek, see for example \cite{roubicek1996}. In this paper the author uses the concept of Generalized Young measures (which is a larger class of measures that includes classical Young measures)  to develop a theory for non-quasiconvex problems. These results focus on integrands whose quasiconvexification is equivalent to their polyconvex envelope. This enables one to set up a relaxation of the problem, RP, and a corresponding discretization, RPd, via Finite Elements. The theory is also able to show existence of solutions to the discretized problem, $(u_d, \eta_d)$, with $u_d$ the minimizer of the relaxation and $\eta_d$ the corresponding generalized measure. Moreover, the author shows that the corresponding sequence of solutions converges to the solution of the relaxed problem, $(u,\eta)$ as the size of the mesh, $d$, goes to zero. Results that continue to build in this direction are in \cite{bartels2004, carstensen2000, chipot1991, kruzik1999, roubicek2011, roubicek2006}.
  
More generally, in higher dimensions the relaxation involves the quasiconvex envelope of the integrand, which is not always easy to find. For this type of problems it is possible to use instead a lower approximation to this object like the polyconvex envelope, or an upper approximation like the rank-one convex envelope \cite{muller1999calculus}. These notions are intimately related to the generalized functional and Young measures. For example, in terms of computational approaches, one can minimize the generalized functionals with additional constraints on the measure. Depending on these constraints one either finds minimizers of an approximate rank-one convexification, see \cite{nicolaides1993}, or as above, minimizers of the polyconvex envelope.

Alternatives to the Finite Element formulation used in the works cited above have also been developed to treat the more manageable case  of energies defined over real valued functions, i.e. $u: \Omega \subset \R^n \rightarrow \R$. Since in this case the measures are supported on a discrete set of points, they can be described as a convex combination of Dirac deltas, see \cite{meziat2006, pedregal1995} and others. This is connected to the fact that for real valued functions the different generalizations of convexity, i.e. rank-one convexity, polyconvexity, and quasiconvexity all coincide. In \cite{meziat2006, meziat2008} these ideas, together with the method of moments \cite{curto2000}, are used to derive an alternative approach for finding the optimal measure. The key point from these papers is that the relaxation can be written in terms of the moments of the measure and the minimization can be recast as a semidefinite programing problem.
 
  We also note that the more direct approach of computing minimizing sequences by directly optimizing the nonconvex energy, has a well developed theory, see \cite{luskin1996} for a review. Of course, with these methods it is not possible to obtain pointwise convergence of minimizers as the mesh is refined. Nonetheless, the results summarized in \cite{luskin1996, luskin1992}, and reference therein, guarantee that nonlinear functionals evaluated at these minimizers converge to the expected values of the probability measures that capture the asymptotic behavior of these solutions. In other words, as the mesh size goes to zero macroscopic quantities evaluated as limits along minimizing sequences. This allows one to compute the microstructure on a larger length scale than the physical length scale. Among the difficulties of this approach is that the mesh's orientation affects the size of the resulting microstructure. 

With the exception of the method of moments, most of the algorithms mentioned in the previous paragraphs treat nonconvex problems using Finite Elements. In contrast, our discretization of the relaxed problem is base on finite differences and a shrink-type operator to solve our minimization. This makes our algorithm very efficient and easy to implement. On the other hand, the disadvantage of our approach is that it does not carry over to energies defined over multivalued functions. 

{\bf Outline:} In the rest of this introduction we go over our notation and the assumptions we make.  In \cref{s:young} we recall key results that show that the relaxation of the functional $I[u]$ through Young measures is indeed given by $\overline{I}$. In \cref{s:controlH} we construct semi-analytic solutions to the relaxed functional using what is known as the control Hamiltonian. Finally in \cref{s:computations}  we describe our modified split Bregman algorithm. We defer the proofs of convergence of our algorithm to \cref{s:AppendixAA}.

{\bf Notation:}
$\Omega \subseteq \mathbb{R}^m$ is a $m$-dimensional domain. We set $\Omega =[a,b]$ except for our final example where $\Omega = [0,1]^2$. We take $u_0$ to be any function in $W^{1,p}(\Omega)$ (resp. \textit{BV}$(\Omega)$ )  that satisfies the desired boundary conditions. In addition, we will denote:
\begin{itemize}
\item The original problem as
\begin{equation}\label{e:original}
 \min_{A} I[u] = \min_A \int_\Omega W(\nabla u) + V(x,u) \;dx,
 \end{equation}
where $ A=\{ u \in W^{1,p}(\Omega) : u-u_0 \in W_0^{1,p}(\Omega)\}$.
\item The generalized problem as
\begin{equation}\label{e:generalized}
 \min_{\mathscr{A}} \tilde{I}[\nu,u] = \min_{\mathscr{A}} \int_\Omega \int_{\R^m} W(\xi) \;d\nu_x(\xi)  + V(x,u) \;dx ,
 \end{equation}
subject to the constraint $\nabla u = \int \xi d\nu_x(\xi)$ and $\mathscr{A} =$ set of all admissible parametrized measures $\nu=\{\nu_x\}$, see \cref{s:young}.
 \item The relaxed problem as
\begin{equation}\label{e:relaxation}
 \min_{A} \overline{I}[u]  =\min_A \int_\Omega \overline{W}(\nabla u) + V(u) \;dx ,
 \end{equation}
where again $A=\{ u \in W^{1,p}(\Omega) : u-u_0 \in W_0^{1,p}(\Omega)\}$ and $\overline{W}$ is the convex envelope of $W$.
\end{itemize}

{\bf Assumptions:} We also make the following assumptions.
\begin{hypothesis}\label{h:main}
 Let $p\geq 2$ and let $f(x,s,\xi)$ denote the integrand
\[ f(x,s,\xi) = W(\xi) +V(x,s).\]
Then:
\begin{enumerate}
\item The function $f(x,s, \xi)$ is a Carath\'eodory function. That is, $f$ is measurable in the variable $x$ and continuous on $(s,\xi)$.
\item Coercivity condition: There are constants  $M, K \geq 0$ and $\alpha >1$ such that
\[  f(x, s,\xi) \geq M |\xi|^\alpha -K.\]
\item Positivity and growth condition: There exists constants $\alpha_1 \in \R$, and  $\alpha_2, \alpha_3 \geq 0$, such that
\[ 0 \leq f(x,s,\xi) \leq \alpha_1+ \alpha_2 |s|^p + \alpha_3 |\xi|^p.\]
\end{enumerate}

\end{hypothesis}

\subsection{Young Measures}\label{s:young}

As we discuss above, our functional $I[u]$ is non-convex and the variational problem~\cref{e:original1} may not have solutions in the classical sense, that is solutions that belong to a Sobolev space. However, by enlarging the set of admissible functions to include solutions described by Young measures we are able to find minimizers for the generalized functional,
\[ \tilde{I}[\nu] = \int_\Omega \int_{\R^m} f(x,u, \xi) \;d\nu_x(\xi)  \;dx, \quad \Omega  \subset \R^m, \]
where the minimization is now over a set of admissible parametrized measures (defined on sets in $\R^m$), $\nu=\{ \nu_x\}_{x \in \Omega}$.  The optimal measure that minimizes the regularized problem is then related to minimizers of the relaxation, $\overline{I}[u]$. In this section we recall the definition of the relaxation, what it means to be an admissible parametrized measure, and state the relaxation Theorem from  Kinderlehrer and Pedregal \cite{kinderlehrer1991} which gives an explicit formula relating minimizers of both, the generalized and the relaxed problem. We then use this information to characterize optimal measures in the one dimensional case and give examples to consolidate all these ideas.

We start by describing the relaxation of a nonconvex functional. For a general minimization problem 
\[ \min_A I[u] = \min_A \int_\Omega f(x,u,\nabla u) \;dx, \quad \Omega \subset \R^m,\]
with integrand $f: \Omega \times \R^n \times \R^{n \times m} \rightarrow \R$, its relaxation is given by
\[ \min_A \overline{I} [u] = \min_A  \int_\Omega Qf(x,u, \nabla u) \;dx,\]
where $Qf$ represents the quasiconvexification of $f$. That is, for a.e. $x \in \Omega$ and for every $(u,\xi) \in \R^n \times \R^{n\times m}$, 
\[ Qf(x,u, \xi ) = \inf \left\{  \frac{1}{|D|} \int_D f(x,u,  \xi + \nabla \phi(y) ) \;dy: \phi \in W^{1,p}_0(D; \R) \right\}, \]
with $D\subset \R$ any bounded open set.
 \begin{Remark}\label{r:convex}
 Since we are working with functionals of real valued functions the quasiconvexification of $f(x,u, \xi)$ is the same as the convex envelope of $f(x,u, \xi)$  in the $\xi$ variable, \cite[Theorem 1.7 p. 10]{dacorogna2007direct}
 \end{Remark}
 
To characterize the set of admissible parametrized measures we first consider the following definition describing a class of parametrized measures.
\begin{definition}
A parametrized measure $\nu=\{ \nu_x\}$ is a $W^{1,p}$-parametrized measure if there is a sequence of gradients $\{ \nabla u_j\}$ such that:
\begin{itemize}
\item $|\nabla u_j|^p$ converges weakly in $L^1$ and  
\item for all $f \in X^p = \{ f \in C(\Omega) : | f(\xi)| \leq C(1 + |\xi|^p\}$ we have $f(\nabla u_j) \rightharpoonup \bar{f}$ in $L^1(\Omega)$ where
 \[ \bar{f}(x) = \int_{\R^{n \times m}} f(\xi) \; d\nu_x(\xi).\] 
\end{itemize}
\end{definition}

With this definition we can now describe the set of {\it admissible measures} $\mathscr{A}$, as the set of $W^{1,p}$-parametrized measures, $\nu$, generated by a sequence of gradients in $W^{1,p}(\Omega)$ subject to 
\[ \nabla u(x) = \int \xi \;d\nu_x(\xi),\qquad u- u_0 \in W_0^{1,p}(\Omega), \]
where $u_0 \in W^{1,p}(\Omega)$ satisfies the required boundary conditions.

Having defined the set $\mathscr{A}$, the characterization of the generalized problem is now complete. In addition, it is well known that if the integrand $f(u, \xi) $ satisfies the following growth conditions
\[ c(|\xi|^p - 1) \leq f(u,\xi) \leq C( 1+ |u|^p +|\xi|^p), \]
then the original problem, its generalization, and its relaxation, all have the same infimum:
\[ \inf_A I[u] = \inf_\mathscr{A} \tilde{I}[\nu] = \inf_A \overline{I}[u]. \]
Moreover, the following Theorem from Pedregal, see \cite{pedregal2012parametrized}, allows us to relate minimizers of $\overline{I}[u]$ to those measures in $\mathscr{A}$ that minimize $\tilde{I}$.

\begin{theorem}\label{t:pedregal}\cite[Corollary 4.6]{pedregal2012parametrized}
Let $\nu$ be a minimizer of $\tilde{I}$. If 
\begin{equation}\label{th:eq1}
 \nabla u(x) = \int_{\R^{n \times m}} \xi \;d\nu_x(\xi),\qquad a.e. \ x \in \Omega, \tag{*}
 \end{equation}
for $u \in W^{1,p}(\Omega)$, then $u$ is a minimizer of $\overline{I}$ and 
\begin{equation}\label{th:eq2}
Qf(x, u ,\nabla u) = \int_{\R^{n \times m}} f(x,u, \xi) \;d\nu_x(\xi), \qquad a.e.\  x \in \Omega. \tag{**}
\end{equation}
Conversely, if $u$ is  minimizer of $\overline{I}$ and $\nu$ is a $W^{1,p}$-parametrized measure such that \cref{th:eq1,th:eq2} hold, then $\nu$ is a minimizer of $\tilde{I}$.
\end{theorem}

For the scalar case $n=1$, the Theorem gives us a method for determining the optimal measure $\nu$ from the minimizer $\overline{u}$, through the expression
\begin{equation}\label{e:mean}
 \overline{W}(\nabla\overline{u}(x)) = \int_{\R^m} W(\xi) \;d\nu_x(\xi), \quad a.e. \; x \in \Omega.
 \end{equation}
Where we used \cref{r:convex} to relate $QW$ to $\overline{W}$, the convex envelope of $W$. Notice as well that we made no assumptions on the function $V(x,u)$, so that these results are equally valid for functionals with potentials which are nonconvex in the variable $u$.

Our task for the rest of this section is to characterize more precisely those measures, $\nu$, that satisfy relation \cref{e:mean}. As shown in \cite{meziat2006}, for $m=n=1$, i.e scalar functions of one variable, it is enough to consider parametrized measures that can be described as the sum of at most two Dirac measures. This follows from the fact that the convex envelope of a function $f: \R \rightarrow \R$ is given by the function $f_e$ whose epigraph is the convex hull of the epigraph of $f$. Then by Carath\'eodory's theorem, any point  on the graph of the convex envelope, $(s, f_e(s))$, can be written as a convex combination of at most two points in the graph of $f$. In other words, one can find two numbers $p_1,p_2$, with $p_1+p_2=1$, and two points $s_1, s_2$ such that
\[(s, f_e(s)) = p_1(s_1, f(s_1)) + p_2(s_2, f(s_2)).\]
This is equivalent to requiring that the convex envelope $f_e$ satisfies
\[ f_e(s) = \int_\R f(\xi) \;d\mu(\xi),\]
where  $\mu$ is the probability measure with mean $s$ and described by $\mu = p_1\delta_{s_1} +p_2\delta_{s_2}$. This idea can also be extended to parametrized measures $\mu = \{ \mu_x\}$, so that for each $x$ we require
\[ f_e(s(x)) = \int_\R f(\xi) \;d\mu_x(\xi).\]
Consequently, the family of parametrized measures that satisfy the relation \cref{e:mean} can described at each $x$ as the sum of at most two Dirac measures. Form this result we can also infer regions of oscillatory behavior. For example, if for each $x$, the optimal measure is described by just one Dirac measure, i.e. $\mu_x = \delta_{u_x(x)}$, then the two problems, $\tilde{I}$ and $ \overline{I}$, are equivalent and the generalized solution is therefore just the function $\bar{u}$, which minimizes $\overline{I}$. On the other hand, if we find that for a particular interval the optimal measure is of the form $\mu = p_1\delta_{u_{x,1}} +p_2\delta_{u_{x,2}}$, then minimizing sequence exhibit oscillatory behavior. Moreover, the probabilities $p_1,$ and $p_2$ represent the fraction of this interval where gradient, $u_x$, is given by $u_{x,1}$ and $u_{x,2}$, respectively.

We end this section with an example that illustrates the ideas from above. Consider the well known Bolza problem
\[I[u] = \int_{-1}^1 (u_x^2-1)^2 + u^2 \;dx, \qquad u(-1) = u(1) =0.\]
It is easy to see that $\inf_u I[u] =0$ and saw-tooth functions with a vanishing amplitude and slopes alternating between $+1$ and $-1$ constitute a minimizing sequence. This sequence generates a Young measure $\mu_x = \frac{1}{2} \delta_1 + \frac{1}{2} \delta_{-1}$. The relaxed functional is 
\[\overline{I}[u] = \int \max(u_x^2-1,0)^2 + u^2 \; dx, \]
whose unique $W^{1,4}$ minimizer is $u = u_x = 0$. This allows us to conclude 
\begin{equation}\label{e:bolzameasure}
0 = \max(u_x^2-1,0)^2 = \int_\R W(\xi) \;d\mu_x(\xi), \quad 0 = u_x(x) = \int_\R \xi\; d\mu_x(\xi),
 \end{equation}
which immediately yields $\mu_x = \frac{1}{2} \delta_1 + \frac{1}{2} \delta_{-1}$, in agreement with the result from the minimizing sequence. Since $\mu_x$ is independent of $x$, the optimal Young measure is spatially homogeneous in this example. In this work, we develop methods that allow us to consider cases where the optimal Young measure $\mu_x$ {\em does depend} on $x$.

\section{Semi-analytic solution to the relaxation via the control Hamiltonian}\label{s:controlH}

In this section we describe a semi-analytical approach for finding minimizers of convex functionals of the form
\[ \overline{I}[u] = \int_\Omega \left[\overline{W}(u_x) + V(u)\right] \;dx, \qquad u-u_0 \in W^{1,p}_0(\Omega), \, \Omega \subset \R. \] 

 The approach comes from optimal control theory and the use of a {\it Control Hamiltonian} \cite{liberzon2012calculus}. In this section we will motivate the use of this method, which allows us to consider functionals or Lagrangians that are not smooth in the gradient $u_x$. This is not a new difficulty. Indeed, this is a feature of optimal control problems where one looks to maximize a revenue function, and where the set of admissible functions must also solve a dynamical system that depends on a time dependent control parameter. The goal is to not only find optimal trajectories, but to also find an optimal control parameter. In general these optimal solutions are not $C^1$, and this in turn implies that the revenue function, which depends on both the trajectory and the control, is also not a smooth function of these variables. As a result one cannot derive Euler-Lagrange equations or rewrite the system in Hamiltonian form.  Thus, to derive necessary conditions for the existence of optimal solutions one needs a more general theory that allows for non-smooth functionals. This is accomplished by the Pontryagin Maximum Principle \cite{Pontryagin2018the}, which provides necessary conditions for the existence of optimal trajectories and controls in terms of a {\it generalized Hamiltonian} \cite{liberzon2012calculus,sussmann1997}. Here the term generalized refers to the fact that this new Hamiltonian depends not only on the state variables and the control, but also on an additional variable called the {\it costate} \cite{liberzon2012calculus} that plays the role of a Lagrange multiplier. Moreover, with this method one makes no apriori  assumptions on the interdependence of these variables.

To make this idea more concrete consider our problem (in Lagrangian  form) \cref{e:original1} with $V(x,u) = V(u)$, and assume for the moment that the Lagrangian is smooth,
\begin{equation}\label{e:minimization}
 I[u] =  \int_a^b L(u,u_x) \;dx. 
 \end{equation}

From the classical theory, the two necessary conditions for a minimizer, $u: [a,b] \rightarrow \R$, of this functional to exist are that the first variation of this functional is equal to zero, i.e. $\dfrac{\delta I}{\delta u} =0$, and that its second variation is positive, i.e. $\dfrac{\delta^2 I}{\delta u^2}\geq 0$. The first condition leads to the Euler-Lagrange equations
\[ \frac{d}{dx} \left(\frac{\partial L}{\partial v}\right)  = \frac{d L}{du},\]
while the second condition can be expressed in terms of the Hessian of $L$, $ \displaystyle \frac{d^2 L }{dv^2} \geq 0.$

We also have an alternative formulation for the first condition via the Hamiltonian. Using the generalized momentum $p = \dfrac{\partial L}{\partial v}$, which is well defined since we are assuming for now that $L$ is smooth in $v$, one can write
\[ \mathcal{H}(u,p) = p \cdot v(u,p) - L(u,v(u,p)).\]
In this formulation the variable $v$ is defined implicitly through the equation for the generalized momentum and is viewed as a function of $u$, and $p$, i.e $v = v(u,p)$. The Euler Lagrange equations can then be expressed as a first order system
\[ u' = \frac{\partial \mathcal{H}}{\partial p} ; \quad p' = - \frac{\partial \mathcal{H}}{\partial u}. \]

The key insight from control theory is that we do not have to make the assumption that $v$ can be expressed as a function of $u$ and $p$. Rather, it is more natural to consider the Hamiltonian $H(u,v,p)$ as a function of these three {\em independent} variables and derive the generalize momentum equation as a necessary condition for the existence of minimizers. Indeed, this is precisely the content of the Pontryagin Maximum Principle, which we paraphrase in this next theorem (see also \cite{liberzon2012calculus,sussmann1997}) --

\begin{theorem}\label{t:controlH}
Given the minimization problem \cref{e:minimization}, define the associated control Hamiltonian as
\[ H(u,v,p) = p\cdot v - L(u,v).\]
If a curve $x \mapsto u(x)$ is a solution to \cref{e:minimization}, then there exists a function $x \mapsto p(x)$ such that the following conditions hold for all $x \in [a,b]$
\begin{enumerate}[i)]
\setlength \itemsep{1.5ex}
\item $ u'(x) = \displaystyle \frac{\partial H}{\partial p}( u,u',p)$
\item $ p'(x) = -\displaystyle \frac{\partial H}{\partial u} ( u, u',p)$
\item $ H(u,u',p) = \max_v H(u,v,p)$
\end{enumerate}
\end{theorem}

Notice that the first two conditions are just a reformulation of the Euler Lagrange equations, and that the last condition can also be expressed as 
\[ \frac{\partial H}{\partial v} = 0; \quad \frac{\partial^2H}{\partial v^2} \geq 0.\]
Moreover, in the case of smooth $L$ this last equation is equivalent to $p = \dfrac{\partial L}{\partial v}$, while the last inequality is the statement $\dfrac{\partial^2 L}{\partial v^2} \leq 0$. In other words, we recover the two necessary conditions for the existence of a minimization based on the first and second variations of $I[u]$. Finally, note that, for a Lagrangian $L(u,u')$, 
the Hamiltonian is a conserved quantity, i.e $\dfrac{d}{dx} H(u(x),u'(x),p(x)) =0$ along solutions.

For us, the principal advantage of using the Pontryagin Maximum Principle is that it allows us to relax the assumption on the smoothness of the Lagrangian $L(u,v)$, by dropping the requirement that $L$ is a smooth function of $v$, while still providing us with a set of conditions for solving the original minimization problem \cref{e:minimization}.


\subsection{Examples}\label{s:examples}

In the rest of this section we illustrate the Pontryagin Maximum Principle with two examples. We refer to the solutions obtained through this method as \emph{semi-analytic} solutions, since in order to arrive at a complete description of minimizers of the relaxed problem $\overline{I}[u]$ we must numerically solve a 
system of ODEs. To tie these results to our previous discussion, we use these solutions to infer the optimal parametrized measure for the corresponding generalized problem $\tilde{I}[\nu]$.

\textit{ Example 1:} Consider the following Bolza problem
\[ I[u] = \int_0^1 (u_x^2-1)^2_* + u^2 \; dx \qquad u(0) = 0, \, u(1) = 1/2,\]
where the potential $(v^2-1)^2_*$ and its convex envelope $(v^2-1)^2_+$ are given by
\[(v^2-1)^2_*= \left \{ \begin{array}{c c c}
\infty & \mbox{for} & v<0,\\[2ex]
(v^2-1)^2 & \mbox{for} & v \geq 0,
\end{array} \right.
\]
\[
 (v^2-1)^2_+= \left \{ \begin{array}{c c c}
\infty & \mbox{for} & v <0,\\[2ex]
1- \frac{4}{3} \sqrt{\frac{2}{3}}v & \mbox{for} & 0\leq v < \sqrt{\frac{2}{3}},\\[3ex]
(v^2-1)^2 & \mbox{for} & v \geq \sqrt{\frac{2}{3}}.
\end{array} \right.
\]

\textit{Semi-analytic solution:} The control Hamiltonian can be written as
\[ H(u,v,p) = pv - (v^2-1)^2_+ - u^2,\] 
and the three conditions in \cref{t:controlH}  take the form of
\begin{equation}\label{e:ode}
 u' = v, \qquad p' = 2u,
 \end{equation}
 \begin{equation}\label{e:maxH}
\frac{\partial H}{\partial v}= p - \frac{\partial}{\partial v} (v^2-1)^2_+ =0, \qquad \frac{\partial^2 H}{\partial v^2} =- \frac{\partial^2}{\partial v^2}(v^2-1)^2_+\leq 0.
\end{equation}

The Hamiltonian, $H$, is not $C^2$ in $v$, but nonetheless $\frac{\partial^2 H}{\partial v^2}\leq 0$ in the sense of distributions. The requirement $\frac{\partial H}{\partial v}=0$ provides us already with a formula for the costate function, $p$, in terms of $v$, which we can then use to write the Hamiltonian in a more useful form:

\[ p(v) = \left \{ \begin{array}{c c c}
\infty & \mbox{for} & v<0,\\[2ex]
-\frac{4}{3}\sqrt{\frac{2}{3}}& \mbox{for} & 0\leq v < \sqrt{2/3},\\[2ex]
4v(v^2-1) & \mbox{for} & v \geq \sqrt{2/3},
\end{array} \right.\]

\[H(u,v,p(v))  = \left \{ \begin{array}{c c c}
\infty & \mbox{for} & v<0,\\[2ex]
-(1+u^2) & \mbox{for} & 0\leq v < \sqrt{2/3},\\[2ex]
3v^4-2v^2-1-u^2 & \mbox{for} & v \geq \sqrt{2/3}.
\end{array} \right.\]

To find the minimizer $u:[0,1] \rightarrow \R$, one can work out that the solution $(u,v,p)$ to \cref{e:ode,e:maxH} that satisfies the boundary conditions, $u(0) = 0$ and $u(1)=1/2$, must have $v(0) =0$. If this were not the case then $p' \neq 0$, which from the expression for $p(v)$ implies that $ v\geq \sqrt{2/3} >0.5$ forcing $u(1) >1/2$. 

Since trajectories travel along level sets of $H$ then  $$H(u(x),v(x),p(v(x))) = H(0,0,p(0)) = -(1+u(0)^2)=-1.$$ Using \cref{e:ode} and the definition for $H$ we infer that $u(x)=v(x)=0$. However, this solution does not satisfy the second boundary condition $u(1)=1/2$, so at some point $x =x^*$ the value of $v$ must jump to $v(x^*)\geq \sqrt{2/3}$. One can again use the fact that the Hamiltonian is a conserved quantity to find that $v(x^*) = \sqrt{2/3}$. Notice that for $v\geq \sqrt{2/3}$ the costate $p$ satisfies $p' = (12v^2-4)v' =2u$, so that we can use the values $u(x^*)=0, v(x^*) = \sqrt{2/3}$ as initial conditions of the dynamical system
\[ u' =v, \quad v' = \frac{u}{6v^2-2}.\]
Finally, to find $x^*$ we integrate this system and require that $u(1)=1/2$. This can be done numerically giving $x^* = 0.4039$. 

If we denote  the solution to the dynamical system by $u^*$, we see that the solution, $\bar{u}$, to the relaxed functional is given by,
\[ \bar{u} = \left \{ \begin{array}{c c c}
0 & \mbox{for} & 0< x< 0.4039,\\
u^*(x) & \mbox{for} & 0.4039\leq x \leq 1.
\end{array} \right.
\]
A plot of the solution is given in \cref{fig:sol1}. A direct computation of the energy of $\bar{u}$ shows that $\inf_{u \in A} I[u] = \overline{I}[\bar{u}] \approx 0.505445$.

\begin{figure}[h] 
   \centering
   \includegraphics[width=2.5in]{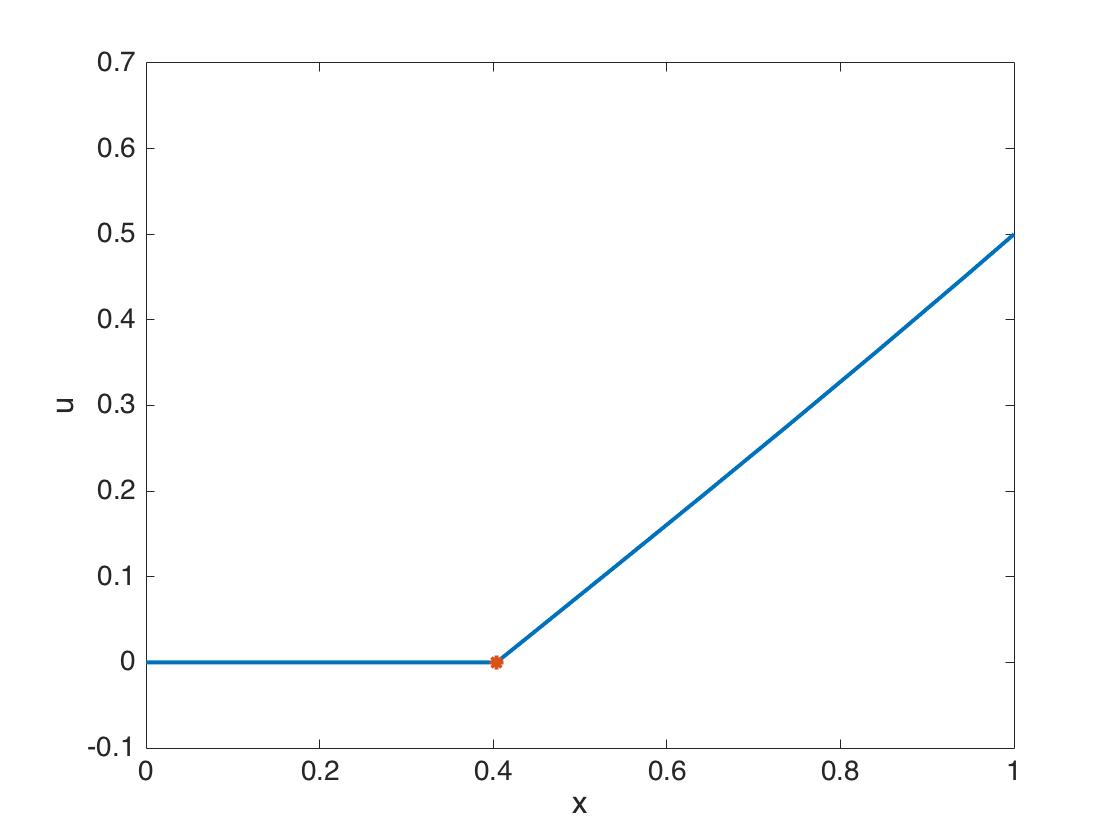} 
   \caption{Minimizer found using the control Hamiltonian for the relaxed problem in Example 1.}
   \label{fig:sol1}
\end{figure}
\begin{remark}

This example was also presented in \cite{meziat2006}, where the authors use a different method for finding minimizers of the relaxation. Starting from the generalized functional in terms of Young measures, they obtain its relaxation by rewriting this integral in terms of the moments of the measure. This leads to an optimization problem that seeks to minimize the relaxed functional over all possible vectors representing the moments of the measure, subject to a matrix inequality that guarantees that the moments come from a non-negative probability measure. Their method leads to the following solution
\[ u(x) = \left\{ \begin{array}{c c c}
0 & \mbox{for} & 0 \leq x \leq 2/5,\\
\frac{5}{6}x - \frac{1}{3} & \mbox{for} & 0< x \leq 1,
\end{array} \right. \]
which is not as precise as our result. Indeed, from the Pontryagin Maximum Principle we know that the Hamiltonian is a conserved quantity. Based on the initial conditions $u(0) = v(0) =0$, we know that the solution must be in the level set $H=1$. A short calculation shows that the solution obtained in \cite{meziat2006} does not stay on this level set.
\end{remark}

\textit{ Young measure result:} We now relate the semi-analytic results to the optimal parametrized measure of the generalized problem,
\[ \tilde{I}[\mu] = \int_0^1 \int_\R (\xi^2-1)^2_* \;d\mu_x(\xi)+ u^2 \; dx, \qquad u(0) = 0 \;u(1) = 1/2.\]

From \cref{t:pedregal}, we know that given a solution, $\bar{u}$, to the relaxed problem, $\bar{u}_x \geq 0$ and the optimal parametrized measure, $\mu$, satisfies $\overline{W}(u_x) = \int W(\xi) \;d\mu(\xi)$. Therefore, for this example the optimal measure is given by 
 \[ \mu_x = \left \{ \begin{array}{c c c}
\lambda(x) \delta_0 + (1- \lambda(x)) \delta_{a} & \mbox{for} & 0\leq u_x(x) < a,\\
\delta_{u_x(x)} & \mbox{for} & a\leq u_x(x),
\end{array} \right.
\]
where $a= \sqrt{2/3}$ and $\lambda(x) = \dfrac{a-u_x(x)}{a}$.

In addition, the optimal parametrized measure satisfies $\bar{u}_x = \int \xi \;d\mu(\xi)$. Since the derivative of $\overline{u}(x)$ is zero on the interval $x \in [0,0.4039)$, for these values of $x$ the measure $\mu_x = \delta_0$. On the other hand, on the interval $x \in [0.4039,1]$ the derivative satisfies $\overline{u}_x(x)> \sqrt{2/3}$ so that for these values of $x$ the measure $\mu_x = \delta_{\bar{u}_x}$. Since the optimal parametrized measured, $\mu_x$, are Dirac measures at each $x$, the solution to the relaxed problem, $\overline{u}$, is also a classical solution to the original problem $I[u]$. For this functional, we can conclude that minimizing sequences do not develop fine-scale oscillations with a nonvanishing amplitude.


\textit{Example 2:} Consider the fully nonconvex Bolza problem,
\[ I[u] = \int_{-1}^1 (u_x^2-1)^2 + (u^2-1)^2\;dx, \qquad u(-1) =0,\quad u(1) =0.\]
Some natural test functions to consider are $u_0(x) = 0$ and $u_{\pm}(x) = \pm (1 - |x|)$ which satisfy $|u'_{\pm}|^2 = 1$ a.e. A direct computation shows that $I[u_0] = 4, I[u_{\pm}] = \frac{16}{15}$.

As before, we want to define a relaxation $\overline{I}$ such that $\min_u \overline{I}[u] = \inf_u I[u]$ and the minimizer of $\overline{I}$ encodes information about the optimal Young measure. Define $G[u]$ as the largest convex functional $\leq I$, it follows that --
\begin{enumerate} 
\item $G$ is coercive, since $ I[u] \geq \frac{1}{2}\int_{-1}^1 (u_x^4 + u^4) \; dx - 6,$
is a bound from below by a convex, coercive function. 
\item $I[u] = I[-u]$ and the maximum of two convex functions is convex, so $G[u] \geq \max(G[u],G[-u])$ and  $G[-u] \geq \max(G[u],G[-u])$ implying $G[u] = G[-u]$.
\item $u_0(x) = 0$ is a global minimum. Indeed, if $u_n$ is a minimizing sequence, so is $-u_n$ and by convexity $G[u_0] \leq \liminf_n \frac{1}{2}\left( G[u_n] + G[-u_n] \right)$. In particular, $u_0 = \frac{1}{2}u_+ + \frac{1}{2} u_-$ implies that $G[u_0] \leq \frac{1}{2} (I[u_+] + I[u_-]) = \frac{16}{15}$. 
\item If $u_n$ is any sequence (possibly with oscillatory microstructure) with uniformly bounded energy $I[u_n] \leq C$, that converges weakly to $u_0$, it follows from the compactness of the Sobolev embedding $W^{1,4}([-1,1]) \to L^4 ([-1,1])$ that we can extract a subsequence (not relabelled) $u_n \to u_0$ in $L^4$ implying that $\liminf_n I[u_n] \geq \int (u_0^2-1)^2 \;dx = 2 > \frac{16}{15} \geq G[u_0]$. 
\end{enumerate}

This argument shows that, the convex envelope of $I$ is {\em not the right object} to capture the limiting energy for weakly convergent sequences. There is a gap between $\liminf_n I[u_n]$ and $G(u_0)$ for sequences $u_n \rightharpoonup u_0$. This argument also suggests that we should compute the lower semi-continuous envelope with respect to weak convergence in $W^{1,4}$, and this functional is given by the partial convexification \cite[Theorem 1.7]{dacorogna2007direct}
\[ \overline{I}[u] = \int_{-1}^1 (u_x^2-1)^2_+ + (u^2-1)^2\;dx, \qquad u(-1) = 0,\, u(1) =0,\]
where we now define
\[ (v^2-1)^2_+= \left \{ \begin{array}{c c c}
0 & \mbox{for} & |v|<1,\\[2ex]
(v^2-1)^2 & \mbox{for} & |v| \geq1.
\end{array} \right.
\]
\textit{Semi-analytic solution:} The relaxed functional is not convex in $u$ and we do not expect to find unique minimizers. Nonetheless, we can write down the control Hamiltonian 
\[ H(u,v,p) = pv - (v^2-1)^2_+ - (u^2-1)^2,\]
and use \cref{t:controlH} to find the necessary conditions that lead to solutions:
\[ u' = v, \quad p' = 4u(u^2-1),\]
\[ \frac{\partial H}{\partial v} = p - \frac{\partial}{\partial v}(v^2-1)^2_+ =0, \qquad \frac{\partial^2 H}{\partial v^2} = -\frac{\partial^2}{\partial v^2}(v^2-1)^2_+ \leq 0.\]
As in the previous example the last condition is always satisfied (distributionally), while the requirement $\frac{\partial H}{\partial v} =0$ gives a formula for the costate function, $p$, in terms of $v$. This allows us to write the Hamiltonian in terms of $u$ and $v$,
\[ H(u,v,p(v) ) = \left \{\begin{array}{c c c}
 -(u^2-1)^2 & \mbox{for} & |v| <1,\\
 (v^2-1)(3v^2+1) - (u^2-1)^2 & \mbox{for} & |v| \geq 1.
 \end{array} \right.
 \]

There are two cases depending on the value of $v$ at the point $x =-1$.  If initially we assume that $|v(-1)|<1$, then the Hamiltonian 
$$H(u(-1),v(-1),p(v(-1))) = -(u(-1)^2-1)^2=-1.$$ Because the Hamiltonian is a conserved quantity, to stay on the level set $H=-1$ we need $v\equiv 0$. This corresponds to the trivial solution $u_0 = 0$ which has energy $\overline{I}[u_0] = 2$. 

If on the other hand $|v(-1)|\geq 1$ then $p' = (12v^2-4) v' = 4u(u^2-1) $, leading to the following dynamical system,
\[ u' = v, \quad v' = \frac{u(u^2-1)}{3v^2-1}.\]
Notice that this is a reversible system, so that if $(u(x),v(x))$ is a solution, then so is $(u(-x),-v(-x))$. 

Here again we have two options, $v< -1$ or $1<v $. In the case when $u' = v(-1)>1$ the function $u(x)$ must be initially increasing. So, there is a point $x^*$ where $u(x^*) =1$ and therefore $v'(x^*)=0$. 

To find the location of $x^*$ we notice that because the value $|v|\geq 1$, the derivative $u' \geq 1$. Integrating $u'$ from $x=-1$ to $x=x^*$ shows that $x^*$ is less than zero. Since the dynamical system is reversible, the solution is even with respect to the $x-$axis. This implies that the solution must satisfy $u=1$ and $v=0$ on the interval $(x^*,0]$ and that for values of $x \in (0,1]$ the solution must mirror what happens in the interval $[-1,0)$, allowing $u$ to satisfy the boundary  condition at $x=1$. In addition, since $u=1$ and $v=0$ on $(x^*,-x^*)$ the solution must lie on the level set $H=0$ and because we jump to values of $|v|\geq 1$, at $x=x^*$ we must have that $v(x^*) = 1$ and $u(x^*)=1$.

 To find the value of $x^*<0$ and the solution on the interval $[-x^*,1]$ we can integrate the above equations using the change of coordinates $y = x-x^*$ together with the initial conditions $u(y=0)=1$ and $v(y=0)=1$ and stopping as soon as $u(y^*)=0$. With this process we find numerically that $x^* = -0.0529$. 
 
 If we denote  the solution to the dynamical system by $u^*$, we can say that the solution, $\bar{u}$, to the relaxed functional is given by,
\begin{equation}\label{e:sol}
 \overline{u}(x)  = \left \{ \begin{array}{c c c}
u^*(-x) & \mbox{for} & -1\leq x \leq x^*,\\
1 & \mbox{for} & x^*< x< -x^*,\\
u^*(x) & \mbox{for} & -x^*\leq x \leq1,
\end{array} \right. 
\end{equation}
where $x^* \approx -0.0529$ and $u^*$ satisfies $|u^*_x|\geq 1 $. A plot of $\bar{u}(x)$ is shown in \cref{fig:sol3}. Computing the energies of $\bar{u}, u_{\pm}$ and $u_0$ yields 
\[
\inf_{u \in A} I[u] = \overline{I}[\bar{u}] \approx 1.0241 < \overline{I}[u_{\pm}] = I[u_{\pm}] = \frac{16}{15} < \overline{I}[u_0] = 2.
\]

\begin{figure}[h] 
   \centering
   \includegraphics[width=2.5in]{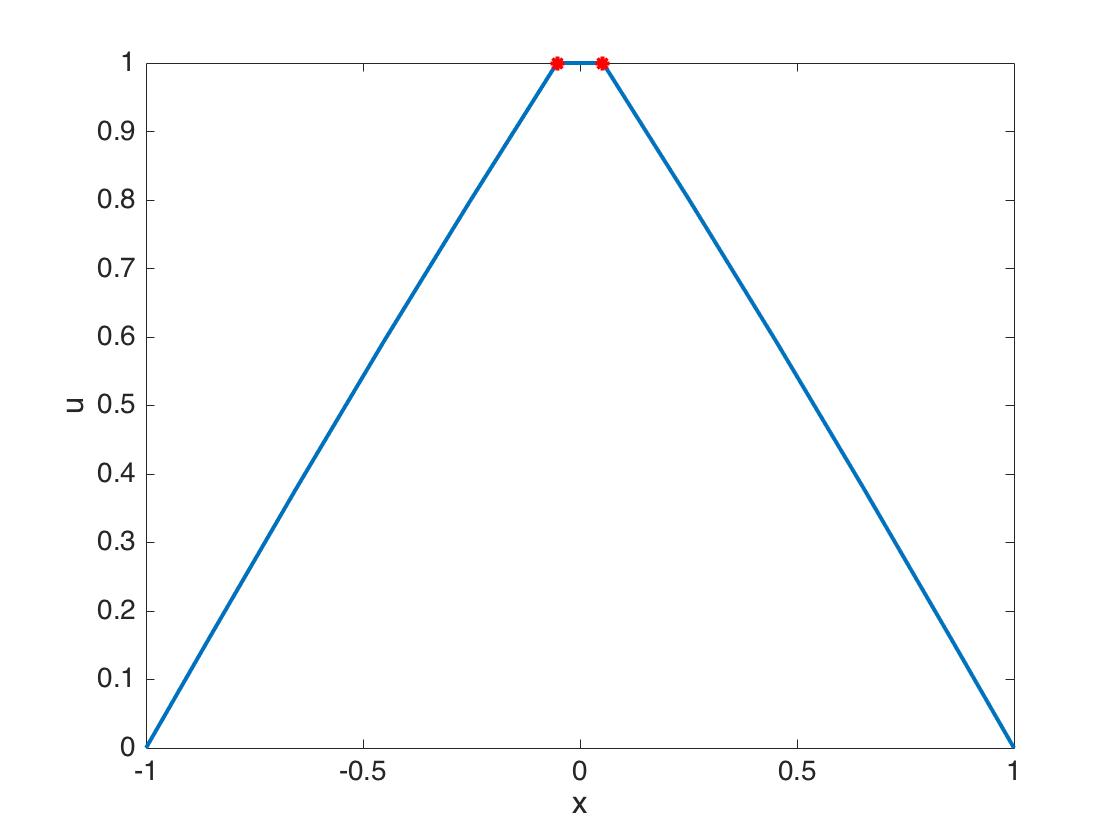} 
   \caption{Minimizer for the relaxed problem in Example 2 from the control Hamiltonian}
   \label{fig:sol3}
\end{figure}
For the second case when $v<-1$, the argument is very similar as the one presented above. The solution in this case is just $-u^*(x)$.

\textit{Young measure:} We now continue by relating the semi-analytic result given by \cref{e:sol} to the generalized functional,
\[ \tilde{I}[\mu] = \int_{-1}^1 \int_\R (\xi -1)^2 \;d\mu_x(\xi) + (u^2-1)^2\;dx \quad u(-1) =0,\; u(1) =0.\]
We know that the optimal parametrized measure must satisfy $$\overline{W}(u_x) = \int W(\xi) \;d\mu_x(\xi),$$ leading to
\[ \mu_x = \left \{ \begin{array}{c c c}
\lambda(x) \delta_{1}+ (1-\lambda(x)) \delta_{-1} & \mbox{if} & |u_x| < 1,\\
\delta_{u_x} & \mbox{if} & |u_x|\geq 1,
\end{array} \right. \]
where $\lambda(x) = \displaystyle \frac{ 1+ u_x}{2}$. Since the optimal measure must also satisfy  $\bar{u}_x = \int \xi \;d\mu(\xi)$, we look at the solution to the relaxed problem we found above.

First notice that for all $ x\in [-1,x^*] \cup [-x^*,1]$ the derivative $|\bar{u}'(x)| >1$, implying that  $\mu_x= \delta_{\bar{u}'(x)}$ on these intervals. On the other hand, for $ x\in (x^*,-x^*)$ we have that $\bar{u}'(x) =0$ and as a result $\mu_x = \frac{1}{2} \delta_{-1} + \frac{1}{2} \delta_{1}$ and we may conclude that minimizing sequences exhibit oscillations on this interval.


\section{Computing the relaxation numerically}\label{s:computations}
While the semi-analytic method from the previous section is fast and very accurate, it is not robust and only applies to problems with special structure. In this section, we propose a robust, problem-independent, numerical scheme for finding minimizers of a (potentially) non-smooth relaxed energy. For notational convenience we reformulate the relaxed variational problem in a more compact form,
\begin{equation}\label{e:relaxed}
\mbox{minimize } \overline{I}[u] := \overline{\mathcal{W}}[u_x] + \mathcal{V}[u], \qquad \mbox{ subject to } u-u_0 \in W^{1,p}_0(\Omega),
 \end{equation}
where  $\overline{\mathcal{W}}[d] = \int_\Omega \overline{W}(d) dx$, and $\mathcal{V}$ is defined analogously. 

We pose the minimization of this functional as a gradient flow problem and look for steady solutions of  $u_t = -\dfrac{\delta \overline{I}}{\delta u}$. This will speed up the convergence of our algorithm, and more importantly it will allow us to incorporate a convex splitting scheme in order to treat the case when the potential $V[x,u]$  is nonconvex. In this latter case, we have to keep in mind that we will be finding local minimizers of $\overline{I}$.

To solve the gradient flow problem we use a modified version of the split Bregman algorithm. Using known properties of this scheme \cite{goldstein2009,osher2005}, we show in \cref{s:AppendixAA} that our algorithm convergences to a minimizer of the discretized relaxed problem. Then, a similar perturbation argument as in \cite{pedregal1996} shows that as the size of the mesh, $h$, goes to zero, the sequence of approximations $u_h$ converges strongly to a minimizer of the relaxed problem. In particular, this means that the solution to the relaxed problem and therefore its associated Young measure is a good approximation of the true optimal measure of the generalized problem, giving a good approximation for the location of microstructures.

We emphasize again that our goal is to use the solutions of the relaxed problem to infer the corresponding Young measure and consequently the location of microstructures. In \cref{s:SBalgorithm} we first review the examples from \cref{s:controlH} and find excellent agreement between the semi-analytic results and the numerical approximations computed using our algorithm. We also find numerical minimizers for  two example problems, examples 4 and 5 below, that do not have an easily computed semi-analytic solution, demonstrating the scope of our algorithm.

\subsection{A modified split Bregman algorithm}\label{s:SBalgorithm}

We first review the split Bregman algorithm \cite{goldstein2009}, which we use here to find minimizers of \cref{e:relaxed}, where both $\overline{\mathcal{W}}[d]$ and $\mathcal{V}[u]$ are convex energy densities.
An equivalent formulation of~\eqref{e:relaxed} is the constrained variational problem
\begin{equation}\label{e:constrain}
 \min_{u,d}  \overline{\mathcal{W}}[d] + \mathcal{V}[u] \quad \mbox{subject to} \quad u_x = d.
 \end{equation}
 We can impose the constraint (approximately) by recasting as an unconstrained problem with a ``large" penalty parameter $\gamma$.
\begin{equation}
 \min_{u,d}   \overline{\mathcal{W}}(d) + \mathcal{V}[u] + \frac{\gamma}{2} \| d- u_x\|^2.
 \label{e:augmented}
 \end{equation}
The advantage, of course, is that $u$ and $d$ are now decoupled, but the drawback is that the resulting variational equations are stiff if $\gamma$ is large and the convergence can be very slow \cite{goldstein2009}. Interestingly, the minimizers of~\eqref{e:constrain} can also be obtained by iterating the following {\em split Bregman} scheme \cite{goldstein2009} (see also appendix~\ref{s:AppendixAA}),
\begin{align}
(u^{k+1},d^{k+1}) & = \mbox{ argmin}_{u,d} \overline{\mathcal{W}}[d] + \mathcal{V}[u] + \frac{\gamma}{2} \| d -u_x -b^k\|^2, \nonumber\\
b^{k+1} & = \ b^k + (u_x^{k+1} - d^{k+1}).
\label{alg:split-Bregman}
\end{align}
The functionals $\overline{\mathcal{W}}$ and $\mathcal{V}$ are decoupled and we can carry out the minimization in two steps,
\begin{align*}
u^{k+1}  =& \mbox{ argmin}_{u,d}  \mathcal{V}[u] + \frac{\gamma}{2} \| d^k - u_x -b^k\|^2, \nonumber \\
d^{k+1} =& \mbox{ argmin}_{u,d} \overline{\mathcal{W}}[d] + \frac{\gamma}{2} \|d - u_x^{k+1}  -b^k\|^2.
\end{align*}
The first subproblem can be solved using for example a conjugate gradient method or Gauss-Seidel, while the second \emph{nonsmooth} subproblem can be solved by a piecewise shrink operator which we define in \cref{e:shrink}.

We remark on a few key features of the split Bregman algorithm~\eqref{alg:split-Bregman}
\begin{enumerate}
\item The update for $b^k$ is {\em not from minimizing} the augmented functional $E^k = \overline{\mathcal{W}}(d) + \mathcal{V}(u) + \frac{\gamma}{2}\|d - u_x - b^k\|^2$ that is defined  in~\eqref{alg:split-Bregman}.
\item $(u^{k+1},d^{k+1})$ are the minimizers of an augmented functional $E^k$. However, the variational equations for $E^k$ {\em are not the same} as those of the objective \eqref{e:constrain}, or the version with the soft constraint \eqref{e:augmented}. In particular, the functionals $E^k$ depend on $b^k$ which varies from one step to the next. Consequently, the energies $\overline{I}[u^k]$ {need not}, and in general do not, decrease monotonically when evaluated on the sequence $u^k$ (See Fig.~\ref{f:EnergyConverges}).
\item The split Bregman iteration has an {\em error forgetting} property \cite{yin2013}. Since the functional $E^k$ changes by an amount that depends on the change in $b^k$, any ``errors" $\|\tilde{u}^{k+1} - u^{k+1}\|$ and $\|\tilde{d}^{k+1} - d^{k+1}\|$ between approximate minimizers $\tilde{u},\tilde{d}$ and the true minimizers of $E^k$ are ``forgotten", once $b^k$ is updated, provided they are smaller than $\|b^{k+1}-b^k\|$. 
\item Under certain ``reasonable" hypotheses on $\overline{\mathcal{W}}$ and $\mathcal{V}$ (see  discussion in appendix~\ref{s:AppendixAA}) we can show that $\|u^k_x - d^k\| \to 0$ (Prop.~\ref{prop:hminimizing}). Prop.~\ref{prop:fixed}  implies that $b^k \to b^*, u^k \to u^*, d^k \to d^* = u_x^*$, a fixed point for the Bregman iteration, which is necessarily a minimizer for the constrained variational problem~\eqref{e:constrain}.
\item In contrast to constrained optimization methods, the split Bregman iteration converges to the minimizer of~\eqref{e:constrain} {\em for any choice} $\gamma > 0$. $\gamma$, therefore, need not be ``large"  and can be chosen to optimize the rate of convergence \cite{goldstein2009}. 
\end{enumerate}

Note that we cannot use the algorithm as formulated above to find minimizers of functionals with $V[x,u]$ nonconvex. As we discuss in the introduction, this can be remedied by recasting the problem as a gradient flow, using a convex splitting scheme, and then adapting the split Bregman algorithm to solve the resulting convex problem. 

In what follows, we will consider evolution in `time' for a gradient flow, as well as split Bregman iterations for minimizing a `time-independent' functional. To keep this distinction clear, we will use a superscript index $u^k$ for the Bregman iterations, and a subscript index $u_n \equiv U_n$ for time evolution.

To describe our method we first review the main ideas behind convex splitting schemes. As the name suggest, these numerical algorithms consist in splitting a nonconvex functional, $\overline{I}$, into a convex part, $\overline{I}_+$, and a concave part, $\overline{I}_-$. The weak formulation of the gradient flow is 
\[ \langle \partial_t u,w \rangle  = - \left( \frac{\delta \overline{I}+}{\delta u}[u],w\right) -\left( \frac{\delta \overline{I}-}{\delta u}[u],w\right),   \]
where $u \in u_0 + H^1_0(\Omega)$ and $\langle \cdot, \cdot\rangle$ is the inner product in $H^1_0(\Omega)$. The contribution of the nonconvex part $\overline{I}_-$ is treated explicitly in the time stepping, i.e. it is evaluated at a previous time step and treated as a forcing term. For a time step of $h$, the algorithm then  consists in solving,
\[ \left\langle \frac{u_{n+1}-u_n}{h}, w \right\rangle  =- \left(  \frac{\delta \overline{I}+}{\delta u}[u_{n+1}],w \right) - \left( \frac{\delta \overline{I}-}{\delta u}[u_n],w \right).\]
This equation is formally the Euler-Lagrange equation for the Rayleigh functional
\begin{align}
    \label{e:Rayleigh}
 R[v; u_n] & = \frac{1}{2h}\langle v-u_n, v - u_n \rangle + \overline{I}_+[v] + \left( \frac{\delta \overline{I}-}{\delta u}[u_n],v-u_n \right) + \overline{I}_-[u_n], 
 \end{align}
where the last two terms are the linearization of $\overline{I}_-$ at $u_n$. Our numerical scheme finds approximate minimizers of the Rayleigh functional $v \mapsto R[v;u_n]$ using the  split Bregman algorithm described above. Since $R$ is strictly convex in its first argument, minimizers exist and are unique. The update rule for the gradient flow is therefore $u_{n+1}  = \arg \min_v R[v;u_n]$.  As shown in \cite{glasner2016}, the sequence $\{u_n\}$, of minimizers of $R[\cdot\,;u_{n-1}]$,  converges to a local minimum of $\overline{I}$ to within an error of $O(h)$.

Since local minimizers for the (potentially) non-convex function $\overline{I}$ can be characterized as fixed points for the mapping $u \mapsto \arg \min_v R[v;u]$, it suffices to compute approximate minimizers $\tilde{u}_{n+1} \approx \arg \min R[v;\tilde{u}_n]$ provided that the sequence $\tilde{u}_n$ converges, $\tilde{u}_n \to u^*$, in a sufficiently strong sense that we can pass to the limit in $R$ to get $u^* =\arg \min_v R[v,u^*]$. 

The objective functional $R[ \cdot \, ; u_n]$ changes with $n$. This fits naturally within a split Bregman iteration framework, since the augmented objective function~\eqref{e:augmented} also changes with $b$. Consequently, all we require is that $|R[v; \partial_x \tilde{u}_{n+1}] - R[v;\partial_x \tilde{u}_n]|$ should be comparable to $\|b_{n+1} - b_n\|$ for all the `candidate minimizers' $v$ at step $n$. This, along with the error forgetting property of the split Bregman iteration will ensure convergence to a fixed point even with the approximate inputs $\tilde{u}_n$.

Our algorithm for finding the local minima of $\overline{I}$, using the modified split Bregman algorithm with convexity splitting, as motivated by the preceding discussion, is given in Algorithm~\ref{algorithm1}. A Matlab implementation of this algorithm is available at \url{https://github.com/gabyjaramillo/Bolza-SplitBregman} \cite{code2019}.

\begin{algorithm}
\caption{Split Bregman with convexity splitting \label{algorithm1}}
\begin{algorithmic}[1]

\STATE \textbf{Preliminary:} Nonconvex potential $W[d]$, interval $d \in [a,b]$,   
\STATE $\overline{\mathcal{W}}[d] \gets$  Beneath and Beyond $(W(d),[a,b],N)$

\STATE \textbf{Inputs:} Tolerance $\mathrm{tol}_1$, step-size $h$, parameter $\gamma > 0$, and SB iterations $K$.

\STATE \textbf{Initialize:} $n \gets 0$, $U_0 = u^0 \gets 0, d^0 \gets 0$, $b^0 \gets 0$ 
\REPEAT

\FOR{$k= 0$ \TO $K-1$}
\STATE $u^{k+1} \gets  \textrm{argmin}_{u} \left[\dfrac{\gamma}{2} \|d^k-u_x-b^k\|_2^2+ \frac{1}{2h} \|u-U_n\|_2^2  +\mathcal{V}_+[u]+(\delta \mathcal{V}_-[U_n],u)\right]$
\STATE $d^{k+1} \gets \textrm{argmin}_{d}\quad  \overline{\mathcal{W}}[d] +  \dfrac{\gamma}{2} \|d - u^{k+1}_x-b^k\|_2^2\quad $ implemented using \eqref{e:shrink}
\STATE $b^{k+1} \gets b^k + (u_x^{k+1}-d^{k+1})$ 
\ENDFOR
\STATE $n \gets n+1$, $U_{n} = u^0 \gets u^K$, $D_n=d^0 \gets d^K$, $B_n =  b^0 \gets b^K$

\UNTIL  $\|D_n - \partial_x U_{n}\|^2_2 \leq \mathrm{tol}_1$ 
\RETURN $U_n$
\end{algorithmic}
\end{algorithm}

In our first step we approximate the convex envelope of $W[d]$ 
following the implementation of the the Beneath and Beyond algorithm in \cite{lucet1997}.

This is followed by a gradient flow loop which minimizes the Rayleigh functional, $R[v;u_n]$, at each step using the split Bregman algorithm. In the examples shown in the next section we use five iterations of this scheme, i.e. we set {\bf $K=5$} in our algorithm. 

Although the proof for the convergence of the algorithm relies on the fact that the sequence of Bregman iterates converge to the minimizer of $R[v;U_n]$ as $K \to \infty$, the numerical algorithm does not need to run the split Bregman scheme to full convergence. It is enough to complete just a few split-Bregman iterations in order to guarantee that the sequence $\|D_n - \partial_x U_n\|_2$ decreases. Conversely, in iterating until an error $\|u^K-\bar{u}\|_2 < \|b^{K}-b^0\|_2$ is obtained, the extra level of accuracy is wasted at the next time step of the gradient flow when the values of $B_n,U_n,D_n$ are updated. We terminate algorithm~\ref{algorithm1} when the error in the constraint falls bellow a chosen tolerance, i.e.  $\| D_n - \partial_xU_n \|^2_2 < \mathrm{tol}$, which is the signature for convergence to a fixed point (see Prop.~\ref{prop:fixed} in the appendix).

As with the original split Bregman algorithm, the minimization of the Rayleigh functional can be carried out as two step process.
\begin{align*}
u^{k+1} = &  \;\mbox{ argmin}_{u}\, \frac{1}{2h} \|u-u^k\|_2^2  +\mathcal{V}_+[u]+(\delta \mathcal{V}_-[u^k],u-u^k) + \mathcal{V}_-[u^k]\\[2ex]
 & \hspace{10ex}   + \frac{\gamma}{2} \|d^k- u_x -b^k\|_2^2,\\[2ex]
d^{k+1} = & \;\mbox{ argmin}_{d}\,  \overline{\mathcal{W}}[d] +  \frac{\gamma}{2} \|d -u^{k+1}_x -b^k\|_2^2.
\end{align*}
To tackle the first subproblem we use Gauss-Seidel iterations to approximate the solution to the corresponding Euler-Lagrange equations. Thanks to the error forgetting property of the split Bregman scheme we don't have to compute this solution to full accuracy, with ten iterations being sufficient. 

To solve the second subproblem, we view the gradient of $\overline{W}$ as piecewise constant function,
\[ 
\partial \overline{W}(d) = \left \{ 
\begin{array}{l c l}
s_-  & \mbox{for} &d< d_0\\
s_i & \mbox{for}& d_{i-1}\leq d < d_i,\\
s_+  & \mbox{for} &d_N< d
\end{array}\right.
\]
where $d_i$ are points where $\partial\overline{W}$ is discontinuous, with $s_- <s_i < s_{i+1} < s_+, 0 \leq i < N$.  The minimization is given by the piecewise shrink operator $S_p$, defined as follows 
\begin{equation}\label{e:shrink}
S_p (i, z, \gamma) =
 \left \{  \begin{array}{ l c l}
z- \frac{s_i}{\gamma} & \mbox{for} & z < d_i+\frac{s_i}{\gamma}, \\
z- \frac{s_{i+1}}{\gamma} & \mbox{for} & z > d_i+\frac{s_{i+1}}{\gamma}, \\
d_i & \mbox{for} & d_i+\frac{s_i}{\gamma}\leq z\leq  d_i+\frac{s_{i+1}}{\gamma},
\end{array} \right.
\end{equation}
with $z = u_x^{k+1}+b^k$. We can allow $s_{\pm} = \pm \infty$ 
in which case $S_p(i,z,\gamma)$ equals $d_0$ for $z < d_0$ and $d_N$ for $z > d_N$. 

Numerical experiments looking at the rates of convergence of algorithm~\ref{algorithm1} for various example functionals and various choices of $\gamma,h$ and $\Delta x$ suggest the heuristic $\gamma \sim h \sim \Delta x$ to obtain the fastest convergence. We henceforth adopt this heuristic in this work. This heuristic can be justified, in part, by the following argument. The Euler-Lagrange equations for the first subproblem can be written abstractly as 
\[ \mathcal{L} u = \gamma u_{xx} - \left(\frac{1}{h} + c(x) \right)  u = f,\]
where the coefficient $c(x)$ depends on our choice of potential $V(x,u)$. For all examples considered here $c(x)$ is always a positive function. Using a centered difference approximation we find that the discretized operator has signature 

\[ \mathcal{L}_d u_i =\frac{\gamma}{\Delta x^2}\left ( u_{i-1} -2 u_i + u_{i+1} \right) - (1/h + c) u_i .
\]

By Gershgorin's Circle Theorem we know that all eigenvalues of the operator must lie in circles  centered at $C_i = -( 2\gamma/( \Delta x)^2 +1/h + c(x_i))$ and of radius $R_i = 2\frac{\gamma}{\Delta x^2}$. This allows one to approximate the condition number of $\mathcal{L}$
as 
\[ K(\mathcal{L}) = \frac{4\gamma/ (\Delta x)^2 + 1/h + c(x) }{1/h + c(x)} \sim \frac{ 4\gamma h}{( \Delta x)^2 (1+c(x)h)},
\]
which suggests that in order to reduce the condition number of the matrix $\mathcal{L}_d$, we must pick $\gamma$ and $h$ so that $\gamma h \sim (\Delta x)^2$, consistent with our heuristic $\gamma \sim h \sim \Delta x$.


\subsection{Examples} 
 We conclude this section with some numerical examples. Unless indicated otherwise $\mathrm{tol}_1 = 1e^{-12}, \gamma =0.01, h =0.01, \Delta x= 2^{-7} \sim 0.0078$, with $K=5$ and 10 iterations of Gauss-Seidel for each iteration of gradient flow.
\begin{figure}[h]
\centering
    \includegraphics[width=2.5in]{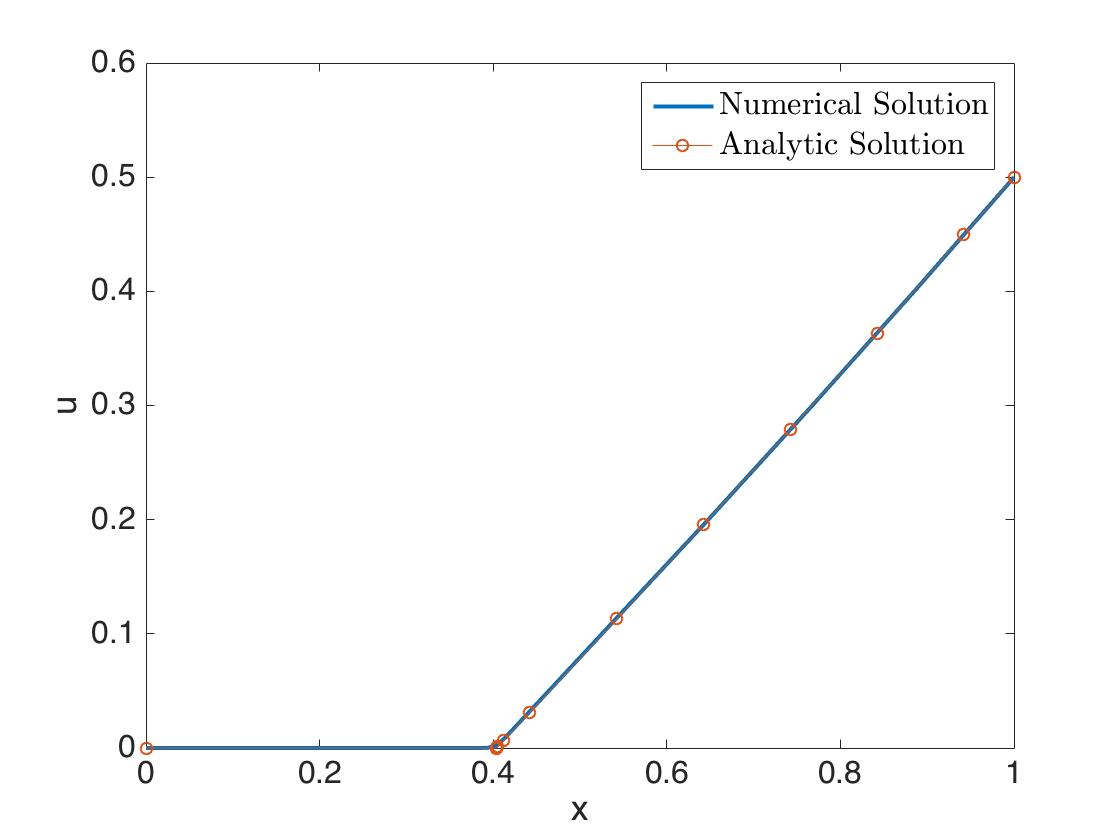}
 \caption{Numerical results for local minimizers of Example 1 using modified split Bregman together with gradient flow, but no convex splitting.  The approximate energy of the solution is $I[\bar{u}] = \overline{I}[\bar{u}]\approx 0.5013$.}\label{f:example1}
\end{figure}

\noindent\textit{Examples 1:}
We again consider the functional
\[ \overline{I}[u] = \int_{-1}^1 (u_x^2-1)_+^2 + u^2 \;dx, \quad u(0) = 0,\, u(1) = 1/2, \]
where $(d^2-1)^2_+$ represents the convex envelope of the piecewise function $(d^2-1)_*^2$, described in \cref{s:controlH}. In \cref{f:example1} the numerical results using the modified split Bregman algorithm are plotted against the analytic solutions found in \cref{s:controlH} showing that they are in excellent agreement. We also confirm that the energy corresponding to the minimizer obtained using our numerical scheme, $\overline{I} = 0.5013$, is in good agreement with the results found using the control Hamiltonian, $\overline{I} =0.505445 $. Towards the end of this section, we show  in \cref{t:energy} the energy, $\overline{I}$, corresponding to various minimizers found using different values of $\Delta x$.

\noindent\textit{Example 2:}
Next we consider a functional which is nonconvex in the variable $u$. 
\[ \overline{I}[u] = \int_{-1}^1 (u_x^2-1)_+^2 + (u^2-1)^2 \;dx \quad u(-1) = 0,\; u(1) = 0.\]
In this example the function $(d^2-1)^2_+$ now represents the convex envelope of the polynomial $(d^2-1)^2$. In \cref{f:example2}, we plot the two global minimizers against their semi-analytic counterpart found in \cref{s:controlH}. Again we find that the energy $\overline{I}$ corresponding to these minimizers is in good agreement with the results from \cref{s:controlH}. To see how the energy converges as $\Delta x$ goes to zero, see \cref{t:energy} at the end of this section.
\begin{figure}[h]
\centering
\subfloat[]
{
    \includegraphics[width=0.45\textwidth]{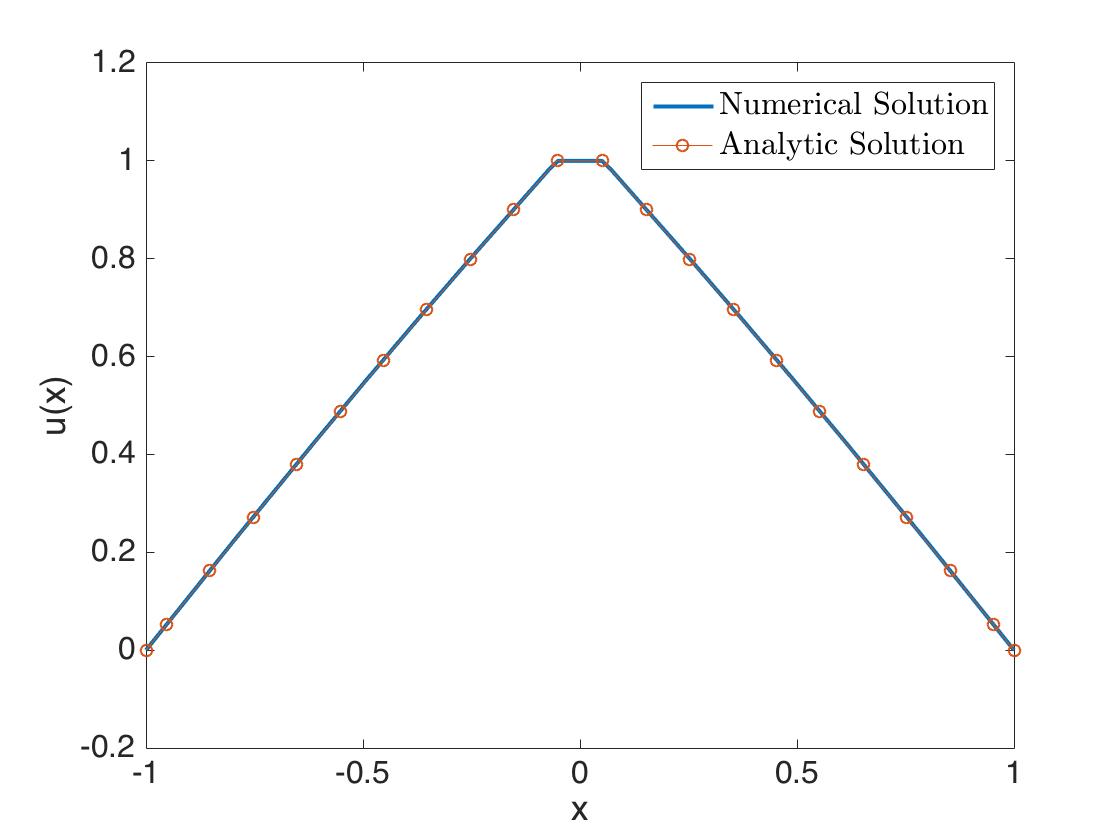}
}
\subfloat[]
{
    \includegraphics[width=0.45\textwidth]{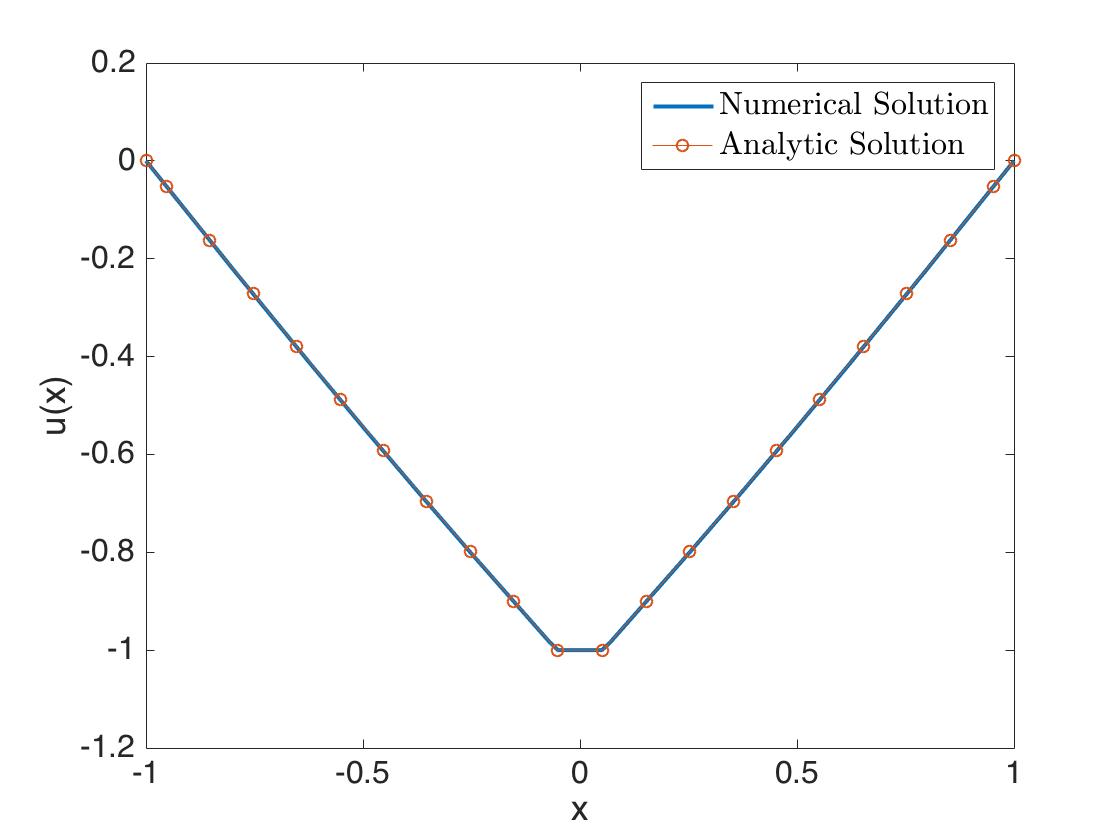}
}
 \caption{Numerical results for local minimizers of Example 2 using the convex splitting: $(u^2-1)^2 = 2au^2 + (u^4 -2(1+a)u^2 +1)$, $a=4$. Figure a) corresponds to an initial guess of $u =1$ and figure b) corresponds to an initial guess of $u =-1$. Both solutions have an approximate energy $ \overline{I}[u] = 1.0234$. Parameters: $\Delta x= 2^{-7}, h=\gamma = 0.01$ }\label{f:example2}
\end{figure}

\noindent \textit{Example 3}:
 We look at a variation of Example 2 with a triple well potential,
 \[ \overline{I}[u] = \int_{-1}^1 [(u_x^2-1)^2( (u_x-2)^2 -1)^2]^{**} + (u^2-1)^2 \;dx, \quad u(-1) = 0,\; u(1) = 0,\]
 where $[(u_x^2-1)^2( (u_x-2)^2 -1)^2]^{**}$ represents the convex envelope of $(u_x^2-1)^2( (u_x-2)^2 -1)^2$. A plot of this potential is given in \cref{f:example3} together with the numerical approximation for two minimizers of this functional.
 \begin{figure}[h]
 \centering
 \subfloat[]
 {
     \includegraphics[width=0.45\textwidth]{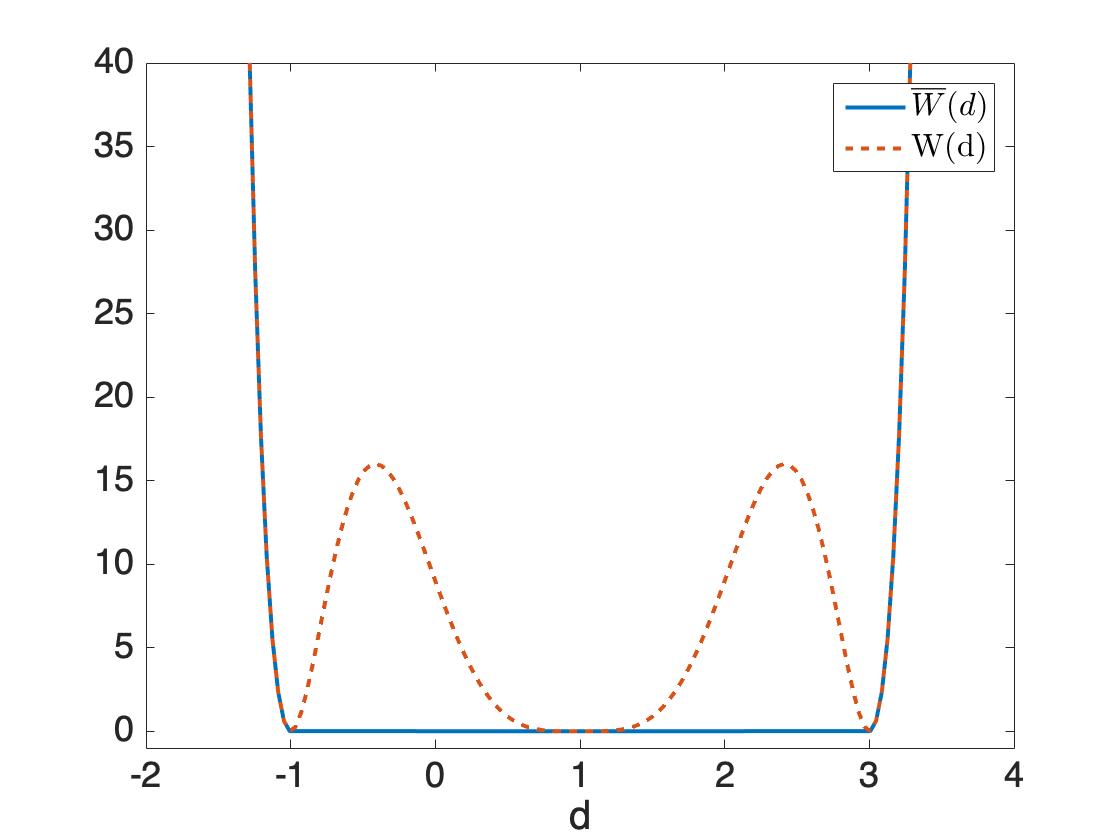}
 }
 \subfloat[]
 {
     \includegraphics[width=0.45\textwidth]{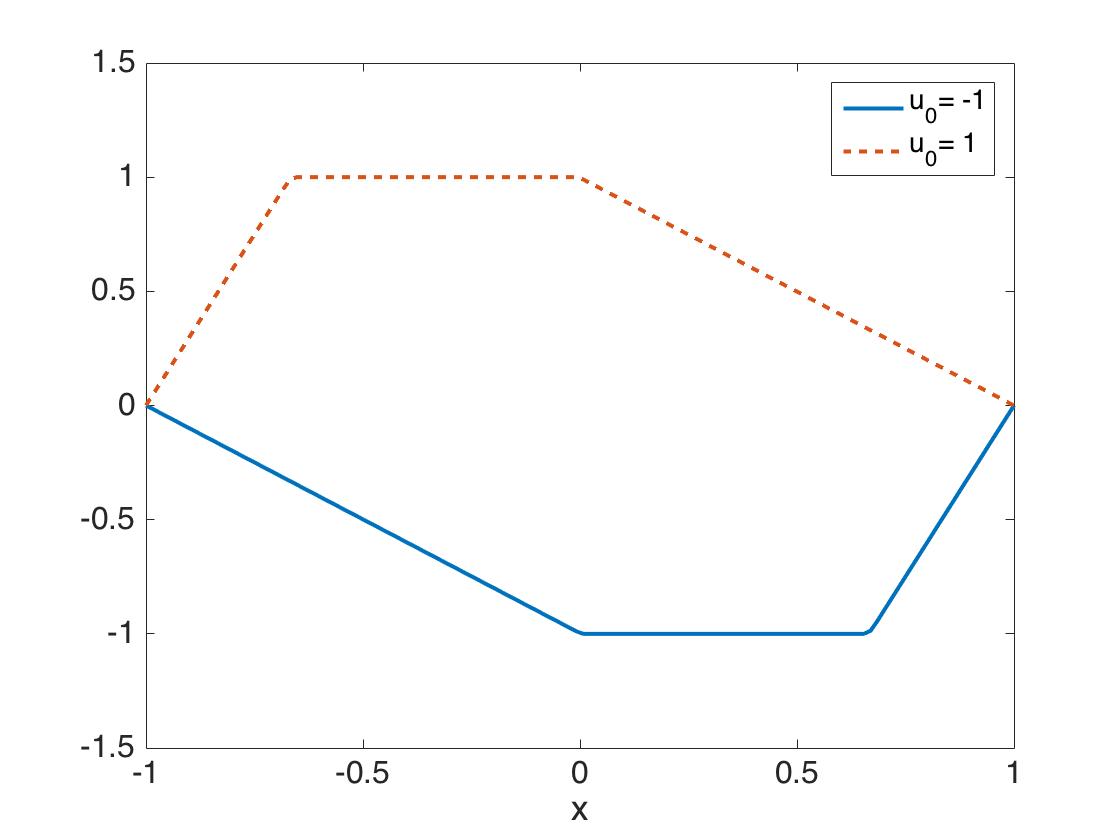}
 
 }
  \caption{a) Convex envelope for $W(d)= (d^2-1)^2( (d-2)^2-1)^2$. b) Two minimizers for the relaxed energy $\overline{I}[u]$ in  Example 3. Both solutions have the same energy. Same convex splitting for $\mathcal{V}[u]$ as in Example 2 with $a = 4$.
  }\label{f:example3}
 \end{figure}
 It is clear from \cref{f:example3} that there are at least two minimizers for this problem with energy $\overline{I}[u]=0.7216$. If we now consider their gradients, which are depicted in \cref{f:gradientE3} one is able to calculate the optimal measure for the generalized problem $\tilde{I}[\nu]$. Labeling the two minimizers of the relaxation as $u_+$ and $u_-$ for the positive and negative solutions, respectively, then their associated Young measures, $\mu_+, \mu_-$ are
 \begin{figure}[h]
 \centering
 \subfloat[]
 {
     \includegraphics[width=0.45\textwidth]{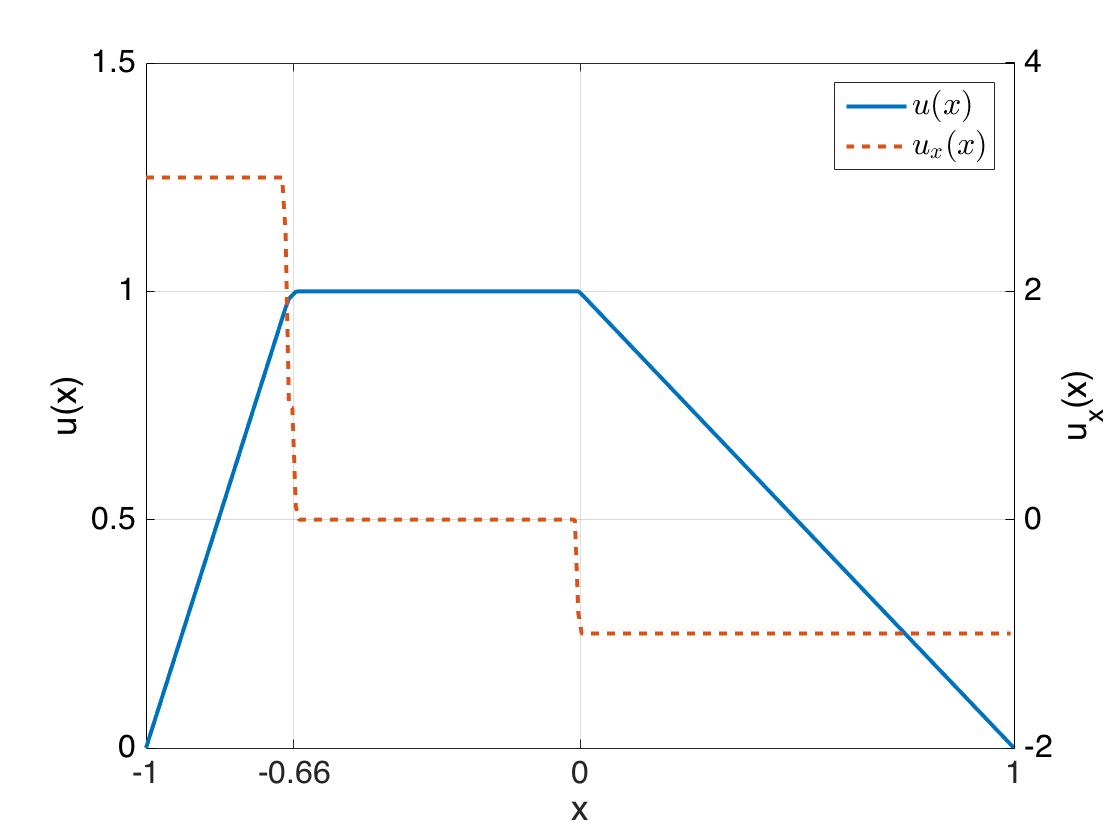}
 
 }
 \subfloat[]
 {
     \includegraphics[width=0.45\textwidth]{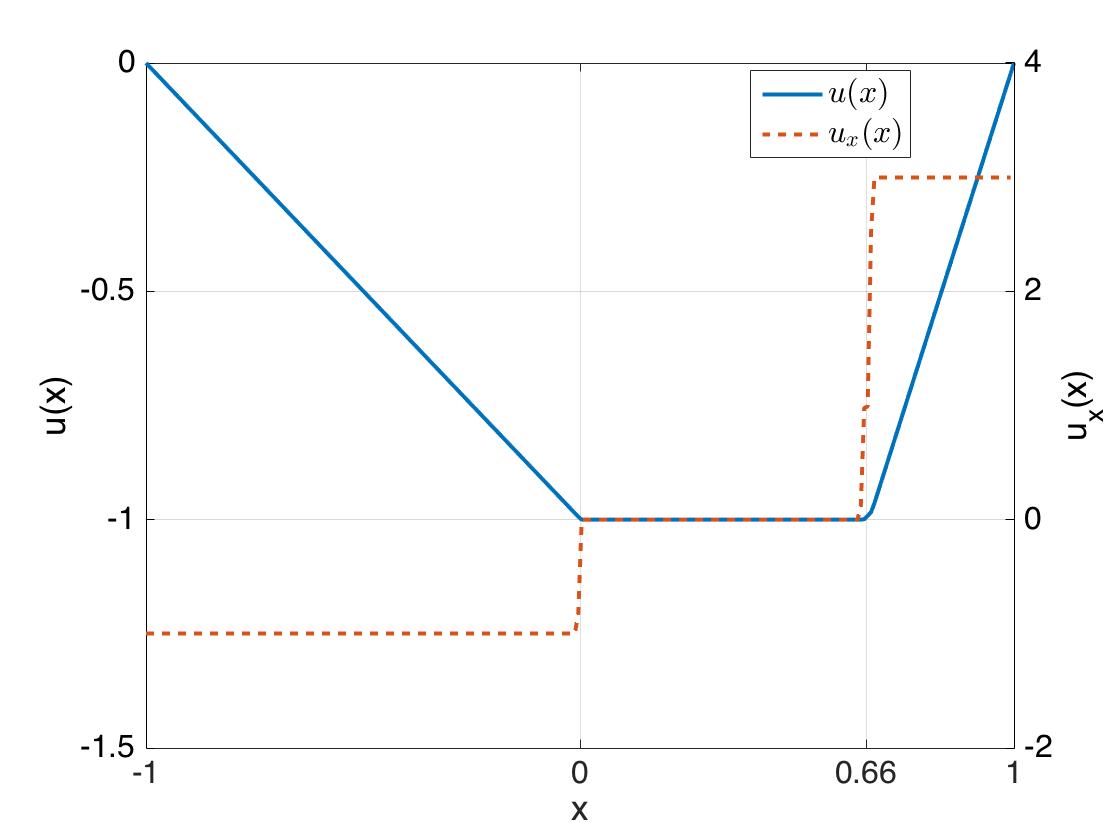}

 }
  \caption{Numerical results for Example 3. By plotting the minimizer, $u$, and its derivative, $u_x$, we are able to visualize the associated measures. Plot a) corresponds to an initial guess $u>0$, and plot b) corresponds to an initial guess of $u<0$.
  }\label{f:gradientE3}
 \end{figure}

 {\small
 \[ \mu_+ = \left \{ 
 \begin{array}{c c c}
 \delta_3 & \mbox{for}& -1\leq x\leq -0.66,\\
 \frac{3}{4} \delta_{-1} + \frac{1}{4}\delta_3 & \mbox{for} & -0.66< x < 0,\\
 \delta_{-1} & \mbox{for} & 0 \leq x \leq 1,
 \end{array} \right.
 \;
 \mu_- = \left \{ 
 \begin{array}{c c c}
 \delta_{-1} & \mbox{for}& -1\leq x\leq 0,\\
 \frac{3}{4} \delta_{-1} + \frac{1}{4}\delta_3 & \mbox{for} & 0< x < 0.66,\\
 \delta_{3} & \mbox{for} & 0.66 \leq x \leq 1.
 \end{array} 
 \right.
  \]
 }

\noindent\textit{Example 4}: We consider the energy $I[u] = \int_{-1}^1 (u_x^2-1)^2 + (u-g(x))^2 \; dx$ for which
 \begin{equation}
     \overline{I}[u] = \int_{-1}^1 (u_x^2-1)^2_+  + (u-g(x))^2 \;dx, 
     \label{e:naturalbc}
 \end{equation} 
with natural boundary conditions. We consider the case when $g(x) = \frac{1}{6}\sin(2 \pi x) +\frac{1}{2} \mathrm{e}^{x}$ and the function $(d^2-1)^2_+$ represents the convexification of the double well potential. In \cref{f:example4a}, we show the minimizer $u(x)$ together with the function $g(x)$ and in another plot we show both $u$ and its derivative $u_x$. We see that $u(x)$ tracks $g(x)$ over part of the interval, and $u_x = 1$ in the complement. We can now infer the Young measure associated with this solution and deduce that minimizing sequences for the nonconvex problem whose relaxation is~\cref{e:naturalbc} should develop oscillatory microstructure on the intervals $(-0.87, -0.3)$ and $(0.3,0.6)$. This feature is not easily predicted before actually solving the relaxed problem. 
\begin{figure}[h]
\centering
\subfloat[]
{
    \includegraphics[width=0.45\textwidth]{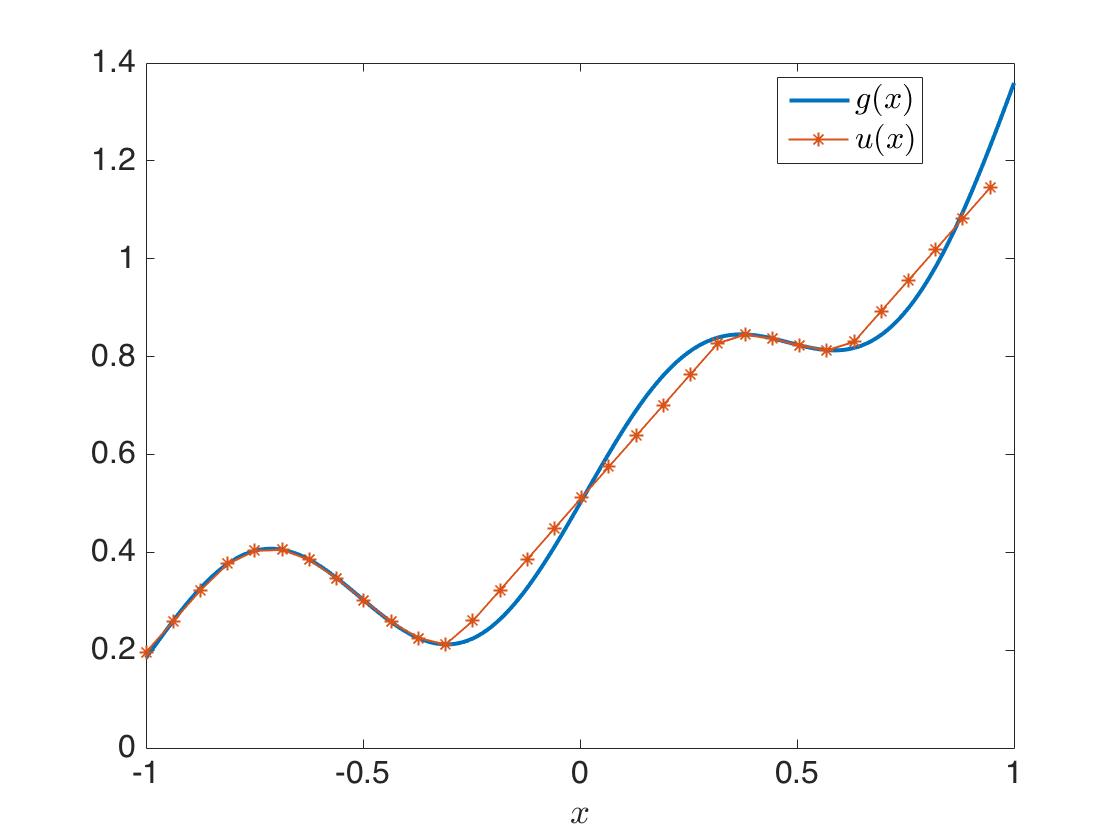}

}
\subfloat[]
{
    \includegraphics[width=0.45\textwidth]{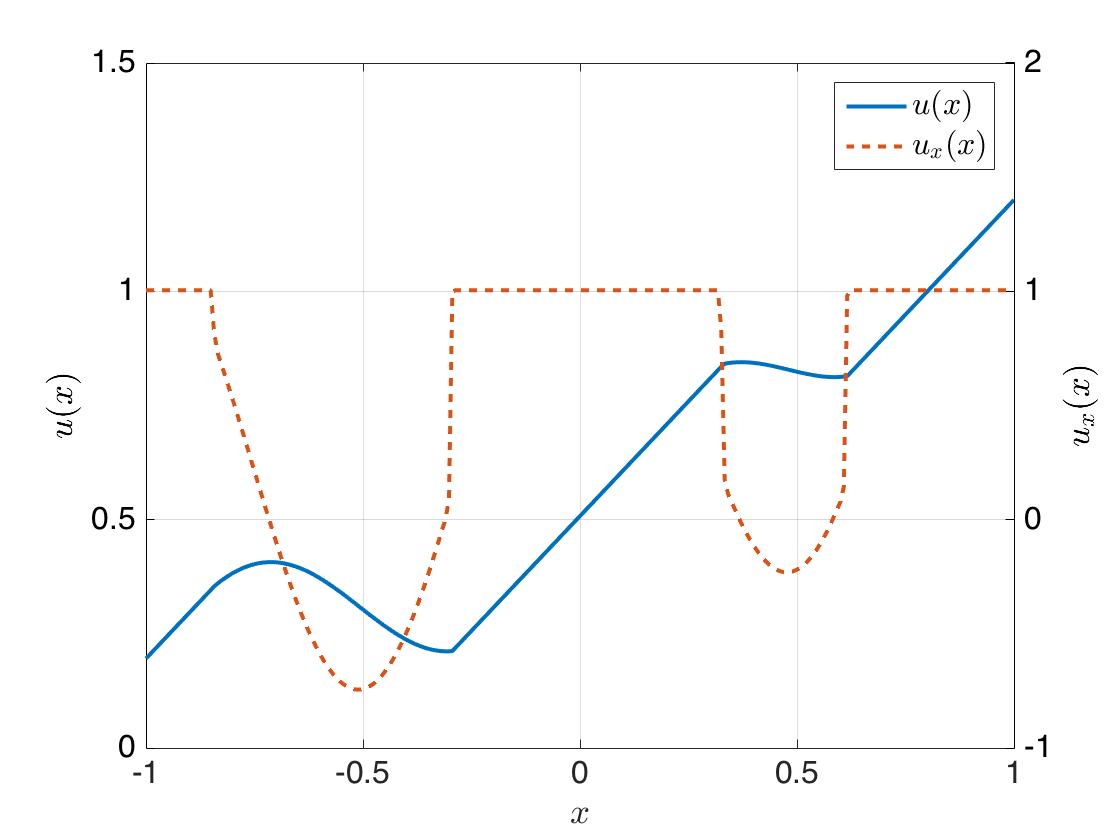}
 }
 \caption{Example 4: a) Plot of minimizer $u$ and the function $g(x)= \frac{1}{6}\sin(2\pi x) +\frac{1}{2} \mathrm{e}^{x}$.  b) Plot of local minimizer $u$ and its derivative $u_x$.  Parameters $\gamma = h = 0.01$ and $\Delta x = 2^{-8}$. For this example we use $K=5$ and 20 iterations of Gauss-Seidel per each iteration of gradient flow. }
 
\label{f:example4a}
\end{figure}

\begin{figure}[ht]
\centering
\includegraphics[width=0.45\textwidth]{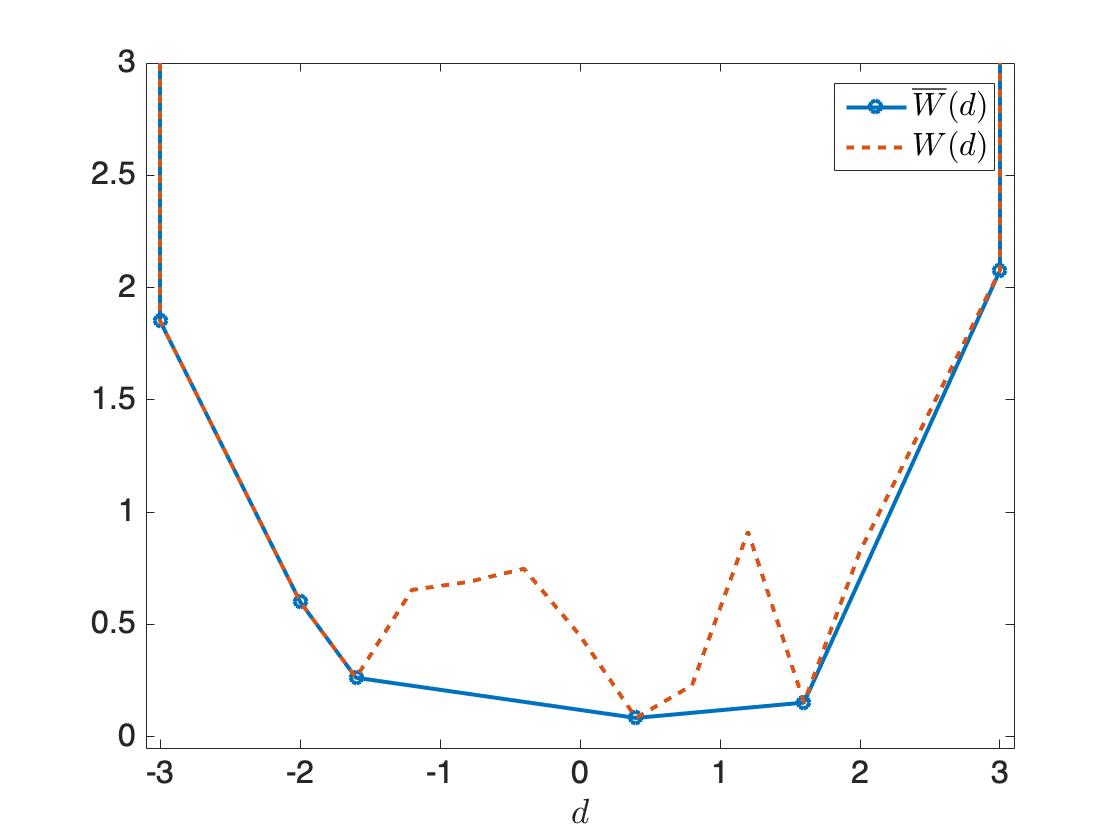}
\caption{ Example 5: Random potential $W[d]$ and corresponding convex envelope, $\overline{W}[d]$.}
\label{f:exampleRand}
\end{figure}

\noindent\textit{Example 5:} We now consider a ``fully numerical" example 
\[ \overline{I}[u] = \int_{-1}^1 \overline{W}[u_x] + (u^2 -g(x))^2 \; dx,\]
with natural boundary conditions, $g(x)= \frac{1}{4}\sin(2\pi x) + \frac{1}{2}$, and $\overline{W}[d]$ the convex envelope of a `random' function.  Here the values of $W(x_i)$ at given points $x_i$ are random samples from a uniform distribution (see \cref{f:exampleRand}). In \cref{f:example6}, we see that the solution $u$ tries to stay close (in absolute value) to the function $\sqrt{g(x)}$, while at the same trying to maintain a slope close to $0.4$. From this minimizer we can infer the associated Young measure and deduce that oscillations will be present in the intervals $(-0.66,-0.37), (0.33,0.62)$.

\begin{figure}[ht]
\centering
\subfloat[]
{ \includegraphics[width=0.45\textwidth]{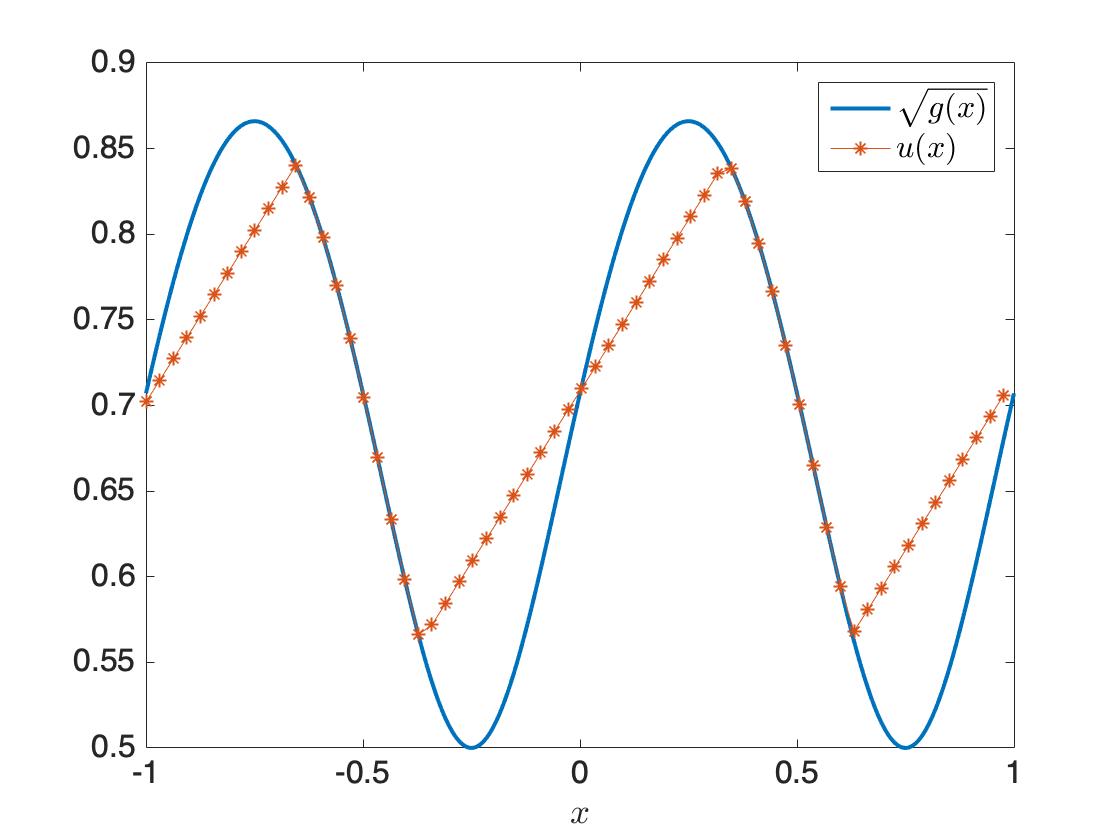}}
\subfloat[]
{\includegraphics[width=0.45\textwidth]{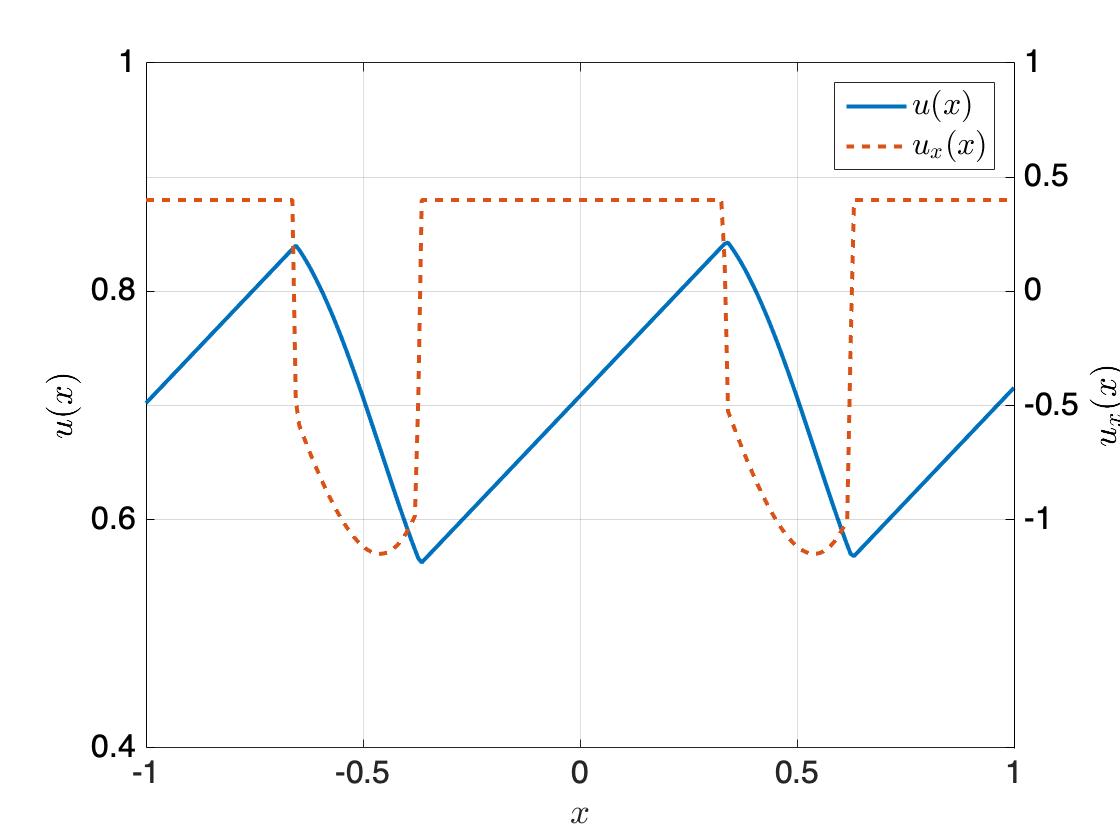}}
\caption{Example 5:  Parameter values used are: $\Delta x= 2^{-8}, h =0.01, \gamma =0.01$. Convex splitting: $(u^2- g)^2 = 2a u^2 + (u^4 -2(g+a) u^2 + g^2), a =4.1$  We also use {\color{blue} $K=10$} and 20 iterations of Gauss-Seidel per each iteration of gradient flow.}
 \label{f:example6}
\end{figure}

In \cref{t:runTime} we record running times for the modified split Bregman algorithm for the different examples presented in this section, and for different values of the grid spacing. Here we set $h=$ max$(\Delta x, 0.01)$ and $\gamma =h$. In  \cref{t:energy}, we also record the energy vs. $\Delta x$ corresponding to minimizers found using our algorithm for the functionals given in Examples 1 and 2. Lastly, in \cref{f:EnergyConverges} we plot the energy $\overline{I}[U_n]$ and the  constraint error $\|D_n - \partial_x U_{n}\|^2_2$ vs. $n$, the number of iterations of the gradient flow, illustrating the fast convergence of the algorithm. We note that the energy and error decay, but not monotonically. The inset shows that the non-monotonicity of the energy persists, albeit on a much smaller scale, even as $t$ gets large. Within each gradient flow step (the outer loop in ~\cref{algorithm1})  we don't need to iterate the split-Bregman steps (inner loop) until $u^k$ converges to the minimizer of the Rayleigh functional $R[\cdot\,;U_{n-1}]$. Precision in $U_n$ beyond the size of $\|b^0-b^K\|$ is ``wasted" \cite{goldstein2009}. In our numerical implementation, we find it sufficient to limit to $K = 5$ split Bregman iterations per gradient flow step.
{\small 
\setlength{\tabcolsep}{6pt}
\renewcommand{\arraystretch}{1.5}
\begin{table}[htbp]
\centering
\begin{tabular}{ c c c c c c }
 $\Delta x $ & {\bf Example 1} & {\bf Example 2} & {\bf Example 3}& {\bf Example 4} & {\bf Example 5} \\
 \hline
 $2^{-5}$ & 0.2414 & 0.5878 & 0.6938 & 0.1439  &0.2369  \\
$2^{-6}$ & 0.6421& 1.1019 & 1.2484 & 0.2782& 0.4538 \\  
$2^{-7}$ & 1.3452 & 2.8495 & 3.5546 & 0.6241  & 0.9542 \\ 
$2^{-8}$ & 1.7376 & 3.4116& 4.5111& 0.8229  & 1.1639 \\ 
$2^{-9}$ & 3.2981 & 3.9084 & 6.3592 & 1.2217 &  1.9043\\
$2^{-10}$ & 7.3790& 10.1703& 16.2610 & 2.4840  & 3.2975
\end{tabular}
\caption{Running times in seconds for our implementation the modified split Bregman algorithm \protect{\cite{code2019}} on a Macbook Pro laptop as measured using Matlab's tic-toc function. Parameters $\mathrm{tol} = 10^{-12}$, $h= \gamma =$ max$(\Delta x, 0.01)$, $K =5$ (except for example 6 where $K =10$), and 10 iterations of Gauss-Seidel.}
\label{t:runTime}
\end{table}
}

{\small
\setlength{\tabcolsep}{6pt}
\renewcommand{\arraystretch}{1.5}
\begin{table}[tbp]
\centering
\begin{tabular}{ c c c c c c c c}
 $\Delta x $ & $2^{-5}$ & $2^{-6}$ & $2^{-7}$ & $2^{-8}$ & $2^{-9}$  & $2^{-10}$ & Semi-analytic\\
 \hline
 {\bf Ex. 1} & 0.4885 &0.4971 &0.5013 & 0.5034 &0.5044 & 0.5049 & 0.50545 \\
 {\bf Ex. 2} &1.0208  &1.0227 &1.0234 &1.0238 &1.0240 & 1.0241& 1.02408
\end{tabular}
\caption{Energies, $\overline{I}[\bar{u}]$, of the minimizers, $\bar{u}$, found using our modified split Bregman algorithm \protect{\cite{code2019}} for different values of $\Delta x$, and using the control Hamiltonian. }
\label{t:energy}
\end{table}
}

 \begin{figure}[htbp]
 \centering
 \subfloat[]
 {
     \includegraphics[width=0.45\textwidth]{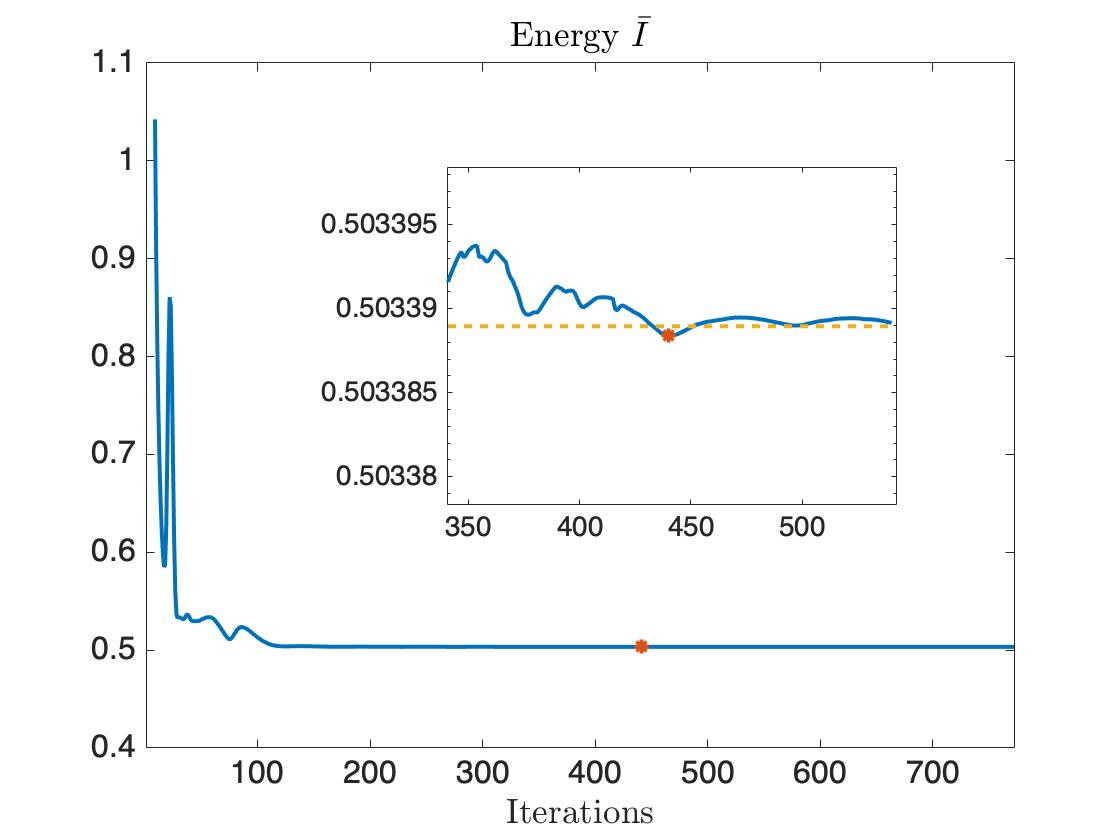}

 }
 \subfloat[]
 {
     \includegraphics[width=0.45\textwidth]{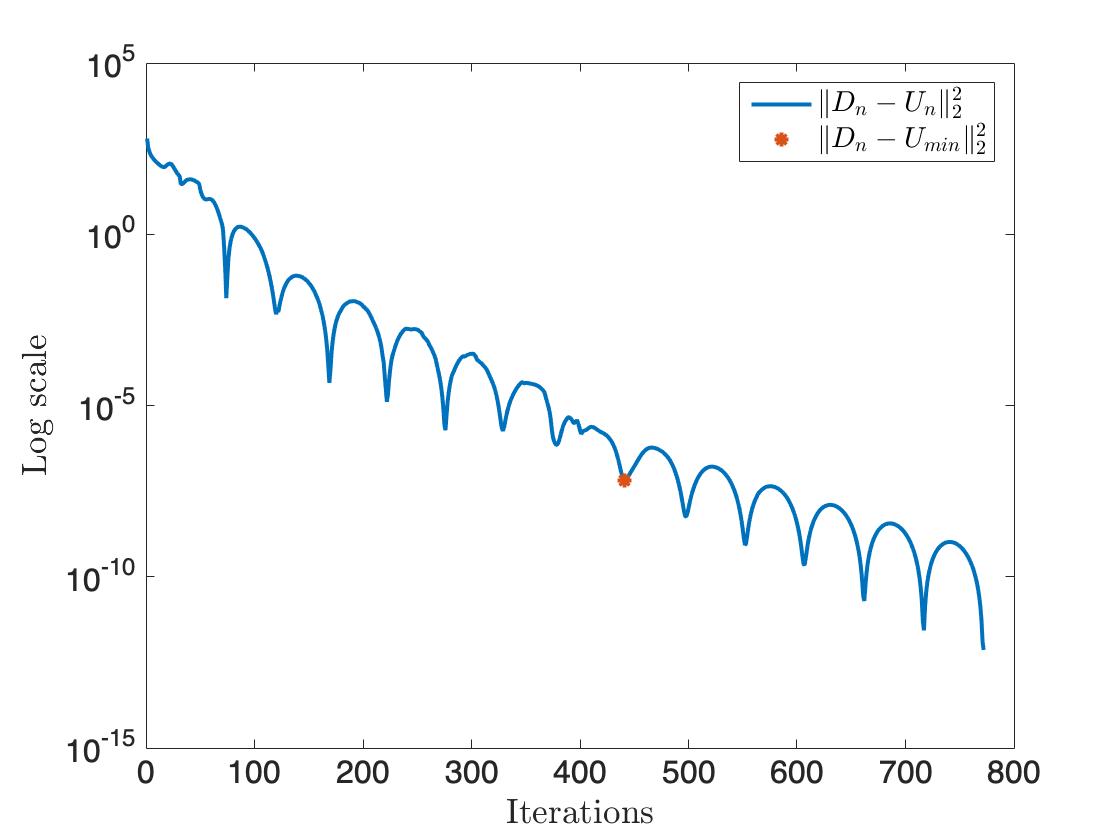}

 }

  \caption{Example 1. Plot of a) the energy $\overline{I}[U_n]$ and b) the $\log_{10}$ of the error $\|D_n - \partial_x U_{n}\|^2_2$ as a function of $n$, the number of gradient flow iterations. The highlighted point in figure a) represents $U_{min}$ where the minimum energy achieved, and the dashed line represents $\overline{I}[U_{last}]$, the energy at the termination of the algorithm. Parameters used in the computations are: $\Delta x = 2^{-8}, h = \gamma = 0.01,$ and $\mathrm{tol} = 1e^{-12}$, $K =5$, and 10 iterations of Gauss-Seidel.}
\label{f:EnergyConverges}
 \end{figure}


\section{Conclusion}
In this paper we develop two methods for finding minimizers of the relaxation of a non-convex energy. We focused on the case of functionals that are defined over scalar valued functions, since for these energies the relaxation involves only the convexification of the energy density with respect to the gradient variable. The issues that we need to resolve include computing the convex envelope (generically non-smooth) and its associated proximal operator numerically, and working with noncovex lower order terms. 

Our first method uses concepts from optimal control theory. We first derive the generalized Hamiltonian for the relaxation of the original nonconvex functional. We analyze this Hamiltonian using the Pontryagin Maximum Principle. This analysis leads to a system of ODEs which give us semi-analytic solutions for the relaxed problem, and thus also for the Young measure associated with minimizing sequences for the original nonconvex energy. 

Our second method is entirely numerical, using modifications of the split Bregman algorithm. Recognizing the similarities between  a piecewise linear approximation of the convex envelope $\overline{W}$ and the $L^1$ norm, we use a split Bregman inspired algorithm to find the minimizers of $\int[\overline{W}(u_x) + V(x,u)] dx$. This energy is analogous to a $L^1$ norm of $u_x$ plus a $L^2$ norm of $u$, a canonical structure for the problems from image processing that motivated the initial development of the split Bregman method \cite{goldstein2009}. There are, of course, substantial differences between problems in image processing, and our motivating problems which come from studying microstructure in materials. These differences include the possibility of a noncovex lower order term $V(x,u)$, which precludes a direct application of methods from convex optimization. We have developed novel strategies to adapt the split Bregman method to these more general problems, for example, by recasting the minimization problem as a gradient flow and using convexity splitting methods.

Our interest in solving the relaxed problem comes from the fact that the nonconvex functionals considered in this paper are connected to their relaxation through the notion of Young measures. This connection allows us to obtain information about the microstructures that arise in the original nonconvex problem. In particular, the Young measure associated with a minimizer of the relaxed problem provides information about the nature and the spatial distribution of microstructure in the original nonconvex problem. We need to justify that the discrete approximations given by   algorithm~\ref{algorithm1} do indeed provide useful information about the microstructures, on {\em scales smaller than the grid spacing},  in  
original nonconvex problem.
  For this justification we recall the results from \cite{pedregal1995} which assert that  if the sequence of approximations, $\{ u_h\}$, converges strongly to a minimizer of \cref{e:relaxation} as the size of the mesh, $h$, goes to zero, then the corresponding sequence of Young measures $\nu_h$ is a macroscopic approximation of the optimal measure of the generalized problem, \cref{e:generalized}. In other words, as long as we have a good approximation to our relaxed problem, then the corresponding Young measure, and consequently the microstructure, are well approximated.

In general, showing the strong convergence of $\{u_h\}$  is difficult and some results in this direction are \cite{french1990, nicolaides1995}. A useful technique is modifying/truncating  the gradients of the sequence $u_h$, to obtain a ``nearby" sequence $\{ \tilde{u}_h\}$ which converges strongly \cite{pedregal1996}. Similar techniques can be used to show that our numerical solution, and the corresponding Young measure, provide enough information to obtain a good approximation of the microstructures present in the original problem.

Although we have largely focused on scalar problems in 1 dimension, the underlying methods are `dimension-independent'. They do, however, rely on computing the quasiconvex-envelope of `gradient' part of the functional. 
This is challenging for  multi-dimensional, vector valued problems, i.e. functionals defined on mappings $u: \Omega^n \subset \mathbb{R}^n \to \mathbb{R}^m$ \cite{muller1999calculus}. On the other hand, our methods extend to functionals defined on vector valued functions of one variable, $u:\Omega \subset \R \rightarrow \R^m$, and multi-dimensional scalar valued functions, $u: \Omega^n \to \mathbb{R}$.  In the latter cases, the quasiconvexification is given by the convex envelope. For vector valued functions,  a generalized Hamiltonian can be found for the relaxed problem along with an equivalent system of ODE. Similarly, the split Bregman algorithm can be extended using a multidimensional shrink operator. Our work along these lines, as well as the connection between these results and a $\Gamma$--development \cite{Anzellotti1993} for the regularized functional $\int [\epsilon^2 u_{xx}^2 + W(u_x) + V(x,u)] dx$, will be presented elsewhere \cite{JV2}. Here we outline a numerical example for minimizing a non-convex functional defined on multi-dimensional scalar functions, using a split-Bregman algorithm along with convexity splitting.

\textit{Example 6}: Minimize $\displaystyle{I[u] = \int_\Omega \left[ \left| \sqrt{u_x^2+u_y^2} - 1 \right| + \frac{1}{4}(1-u^2)^2\right] dx dy}$ over \textit{BV} functions $u: \Omega \to \mathbb{R}$ where $\Omega = [-1,1]^2$ and $u(x,y) = u_0(x,y) = x y$ on $\partial \Omega$.  

$I$ is nonconvex in the gradient $(u_x,u_u)$ and the lower order `Allen-Cahn' term $(1-u^2)^2$ is nonconvex in $u$. The quasiconvexification is obtained by taking the convex envelope of the gradient term \cite{dacorogna2007direct}, to yield
$$
\overline{I}[u] = \int_\Omega \left[ \left(\sqrt{u_x^2+u_y^2} - 1\right)_+ + \frac{1}{4}(1-u^2)^2\right] dx dy.
$$
As before, we use the convexity splitting $\displaystyle{\frac{1}{4}(1-u^2)^2 = \frac{1}{4} + \frac{a}{2} u^2 -\left((1+a) \frac{u^2}{2} - \frac{u^4}{4}\right)}$. The final ingredient is a multi-dimensional shrink operator \cite{goldstein2009,JV2} that computes 
$$
\displaystyle{\arg\min_{d^x,d^y} \left(\sqrt{{d^x}^2+{d^y}^2} -1 \right)_+ + \frac{\gamma}{2} (d^x - \partial_x u -b^x)^2 + \frac{\gamma}{2} (d^y -\partial_y u- b^y)^2 }.
$$ 
We discretize our domain using a square grid with uniform spacing $\Delta$. In our split-Bregman routine we find it optimal to do one Gauss-Seidel step per each time step \cite{goldstein2009}. As per our heuristic, we choose the Bregman parameter $\gamma$, the spatial discretization $\Delta$ and the time step $h$ to be equal to each other $\gamma = h = \Delta$. 
\begin{figure}[htbp]
 \centering
 \subfloat[]
 {
     \includegraphics[width=0.45\textwidth,trim={2cm 0cm 2cm 0cm},clip]{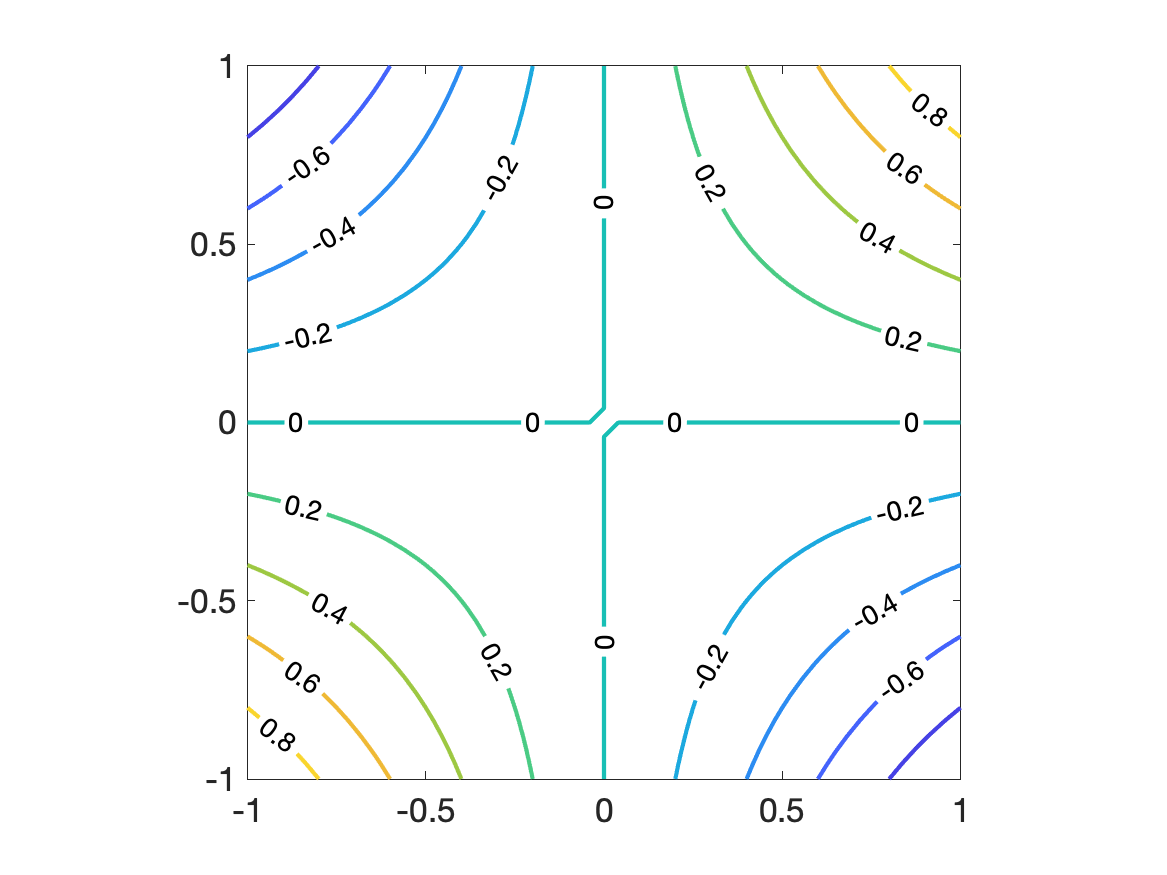}

 }
 \subfloat[]
 {
     \includegraphics[width=0.45\textwidth,trim={2cm 0cm 2cm 0cm},clip]{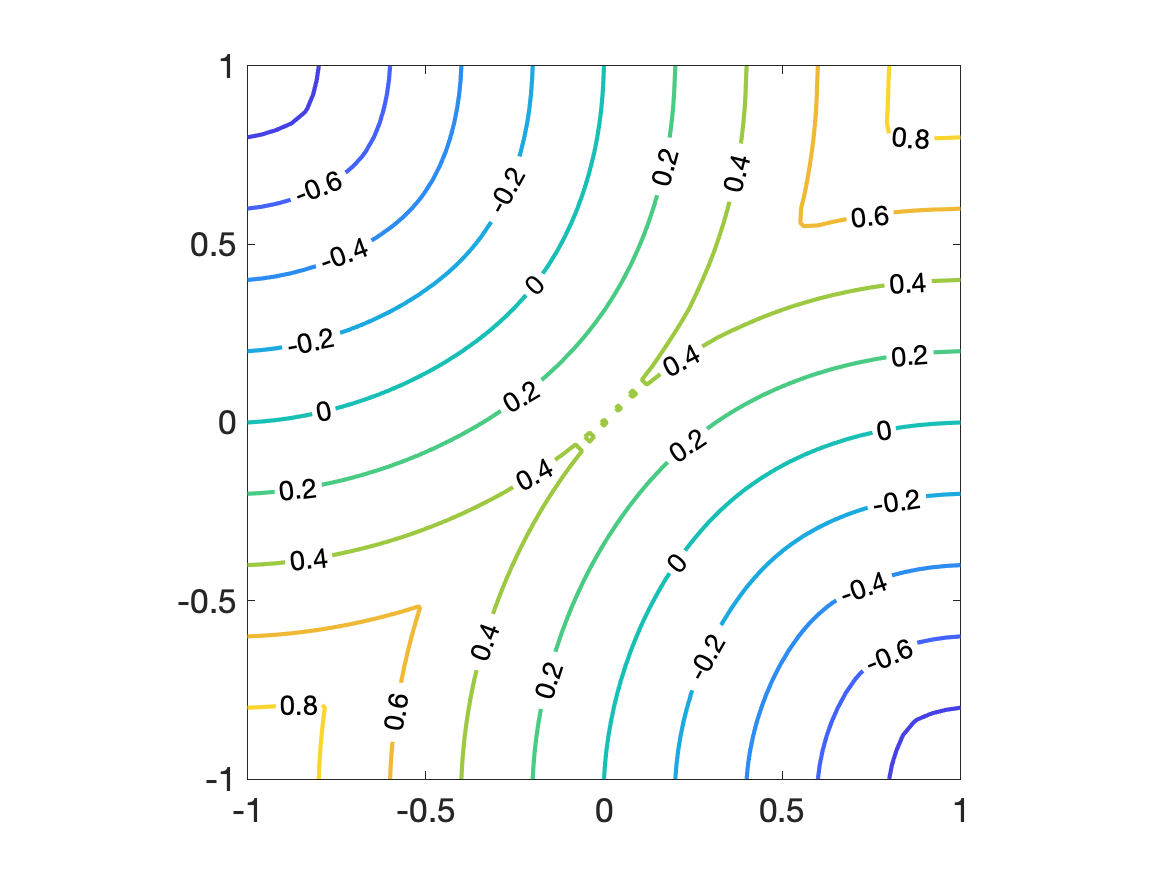}

 } \\
 \subfloat[]
 {
     \includegraphics[width=0.45\textwidth]{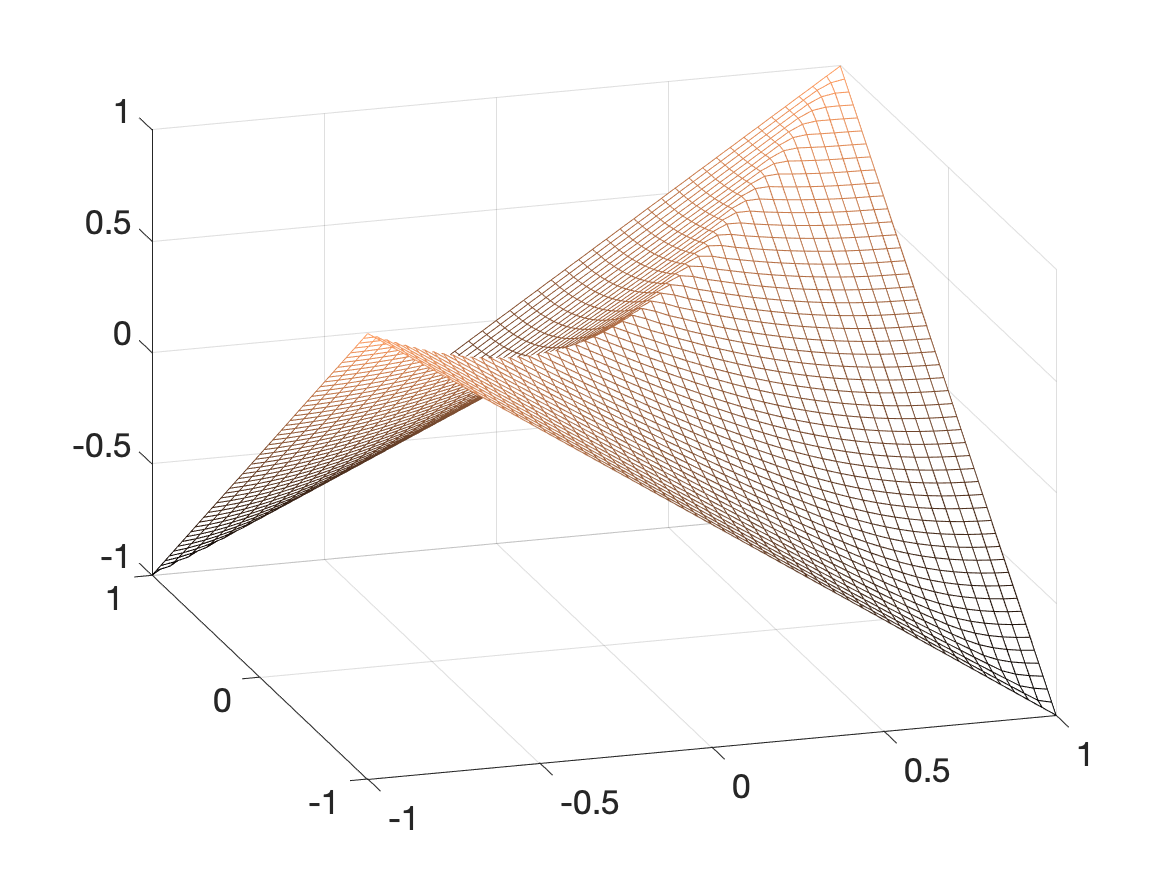}

 }
 \subfloat[]
 {
     \includegraphics[width=0.45\textwidth]{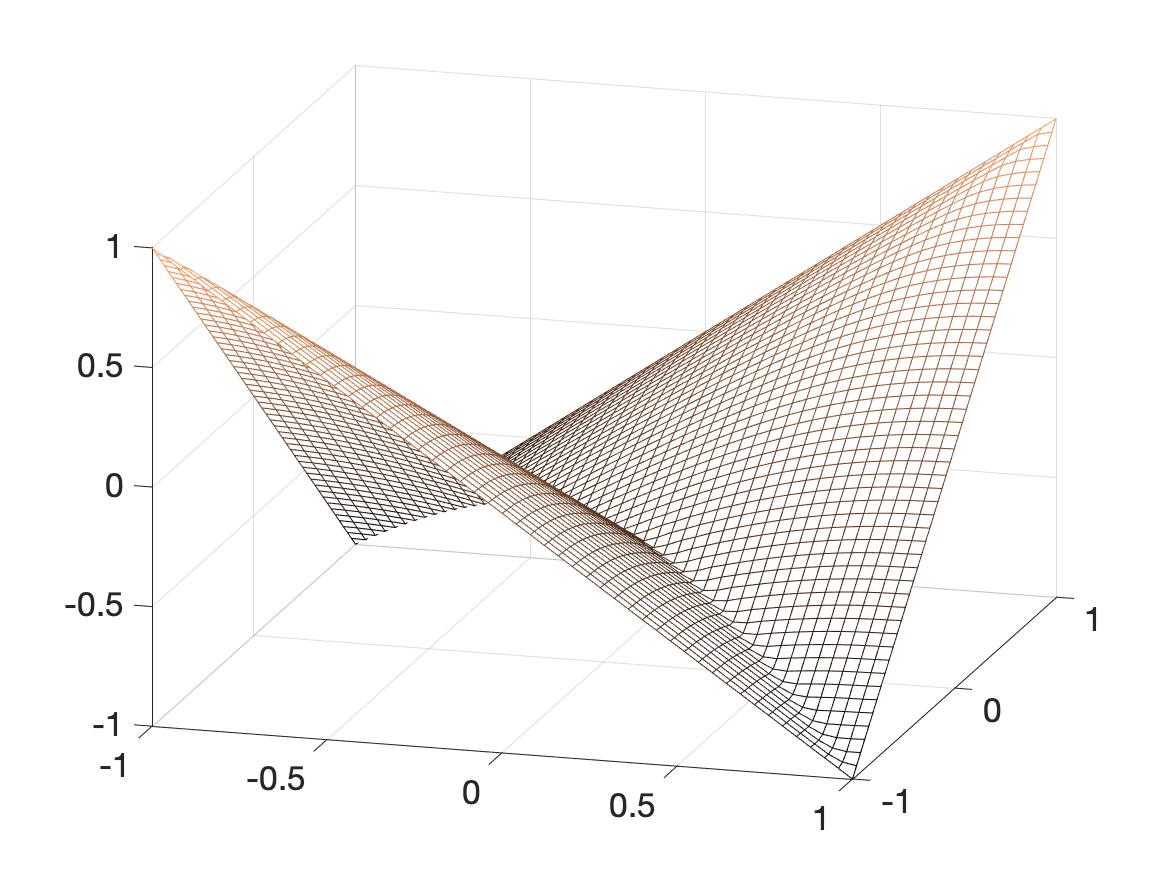}

 }
  \caption{Numerics for Example 6. (a) Contours of the initial condition $u_0(x,y) = xy$  (b) Contours of the minimizier $u_1(x,y)$ shown in (c). The energy $\bar{I}$ is nonconvex because of the lower order terms, and there are two energy minimizers, shown in (c) and (d) that are related by $u_2(x,y) = - u_1(x,-y)$.  The domain is $[-1,1]^2$ and parameters used in the computations are: $\Delta  = h = \gamma = 0.04,$ and $a =2.5$.}
\label{f:2dTVAllen-Cahn}
 \end{figure}
 
Putting everything together, the update for $u$ is
\begin{align*}
u^{k+1}_{i,j}  = & \frac{u^k_{i,j} + h\left[(1+a)u_{i,j}^k - (u_{i,j}^k)^3\right]}{1 + 4 \gamma h/\Delta^2 + a h} + \left(\frac{\gamma h}{\Delta^2 + a h \Delta^2 + 4 \gamma h} \right)\times \\
& \left[ u^{k+1}_{i-1,j} + u^{k+1}_{i,j-1} + u^k_{i+1,j} + u^k_{i,j+1} + \Delta(d^{x,k}_{i,j} - d^{x,k}_{i+1,j} + d^{y,k}_{i,j} - d^{y,k}_{i,j+1}) \right.\\
& - \left. \Delta(b^{x,k}_{i,j} - b^{x,k}_{i+1,j} + b^{y,k}_{i,j} - b^{y,k}_{i,j+1}) \right].
\end{align*}
The update for $d$ is given by a multidimensional shrink operator:
\begin{align*}
\nu^{x,k+1}_{i,j}  = & \frac{u^{k+1}_{i+1,j} - u^{k+1}_{i,j}}{\Delta} + b^{x,k}_{i,j}, \quad
\nu^{y,k+1}_{i,j}  =  \frac{u^{k+1}_{i,j+1} - u^{k+1}_{i,j}}{\Delta} + b^{y,k}_{i,j}, \\
\rho^{k+1}_{i,j} = & \sqrt{(\nu^{x,k+1}_{i,j})^2 + (\nu^{y,k+1}_{i,j})^2}, \\
(d^{x,k+1}_{i,j}, d^{y,k+1}_{i,j})= & \max\left(1-\frac{1}{\gamma\rho^{k+1}_{i,j}},\min\left(1,\frac{1}{\rho^{k+1}_{i,j}}\right)\right) (\nu^{x,k+1}_{i,j}, \nu^{y,k+1}_{i,j}).
\end{align*}
Note that $(d^x,d^y) = (\nu^x,\nu^y)$ for $0 \leq \rho \leq 1$, so there are no computational issues with overflow/underflow. The update for $b$ is given by ``adding back the noise" \cite{osher2005,goldstein2009}
\begin{align*}
(b^{x,k+1}_{i,j},b^{y,k+1}_{i,j}) & = (b^{x,k}_{i,j},b^{y,k}_{i,j}) + \left( \frac{u^{k+1}_{i+1,j} - u^{k+1}_{i,j}}{\Delta} -d^{x,k}_{i,j}, \frac{u^{k+1}_{i,j+1} - u^{k+1}_{i,j}}{\Delta} - d^{y,k}_{i,j} \right).
\end{align*}
Our numerical results are shown in Fig.~\ref{f:2dTVAllen-Cahn}.

Example 6 illustrates the application of our method for multi-dimensional problems in mechanics and microstructure formation. In a related vein, Zhou and Bhattacharya  \cite{Zhou2020} have developed an alternative method for multi-dimensional problems, that also employs a decoupling between the field $u$ and its gradient $\nabla u$. Their method uses the alternating direction method of multipliers (ADMM) in contrast to our approach using the split-Bregman method. Their method is parallelizable and uniquely suited to implementation on GPUs \cite{Zhou2020}. It will be interesting, for future work, to develop similar, parallelizable algorithms based on our methods. 


\appendix
\section{Convergence of the modified split Bregman algorithm} \label{s:AppendixAA}

Here we restate known results about the split Bregman algorithm \cite{osher2005, yin2008, goldstein2009} and adapt them to our setting. For convenience we use the following notation:  $U = (u,d) \in X$, $E(U) = \overline{\mathcal{W}}(d) + \tilde{V}(u)$, where again $\overline{\mathcal{W}}(d)$ is the convexification of $\int W(d) dx$, and $\tilde{V}$ is either equal to $\mathcal{V}$ if this potential is convex, or it is equal to $\mathcal{V}_+[u] + ( \delta \mathcal{V}_-[u^k], u - u^k) + \mathcal{V}_-[u^k]$ if we are using a convex splitting. We also consider the linear operator $BU =d- \partial_x u $ with the corresponding functional $H(U) = \frac{\lambda}{2} \|BU\|^2 $, and the corresponding (penalized) unconstrained variational problem
\begin{equation}\label{e:unconstrained}
 \min_U F(U) = \min_U E(U) + H(U). 
 \end{equation}
The main goal of this section is to show that  sequence of iterates generated by the split Bregman algorithm
converges to the solution of the original constrained problem,
$ \min_U E(U) \quad \mbox{subject to} \quad u_x = d$ or equivalently $H(U) = 0$.

In other words, the following results show that the modified split Bregman scheme, and consequently each iterate in our 'gradient flow' algorithm, is well defined. From this we can conclude that the solution, $u_h$, we obtain from our numerical scheme is indeed a minimizer of the discretized version of the relaxed problem, \cref{e:relaxation}. 

To accomplish this task we will need to consider the two algorithms presented in \cref{t:algorithms}, where the term $D^{P^k}_E(U,U_k)$ represents the Bregman distance given by 
$$ D^{P^k}_E(U,U_k) = E(U) - E(U^k) - \langle P^k, U -U^k\rangle.$$
\begin{table}[h]

    \begin{minipage}{0.48\linewidth}
    { \footnotesize
    \begin{tabular}{ | p{1.1cm} p{4.25cm}  |}
\hline
\multicolumn{2}{|c|}{ \bf Bregman Iteration}\\[2ex]
\hline
\hline
 &  $U^0 = 0 $\quad $P^0=0 $\\[3ex]
$U^{k+1} =$ & argmin$_U \;D^{P^k}_E(U,U^k) + \frac{\lambda}{2} \| BU\|^2$ \\[2ex]
 $P^{k+1} = $ & $P^k - \lambda B^TBU^{k+1}$\\[2ex]
 
\hline
\end{tabular}
 \label{t:split}}

 \end{minipage}
 \qquad
 \begin{minipage}{0.45\linewidth}
   
    {\footnotesize
    \begin{tabular}{| p{1.1cm} p{3.9cm}|}
\hline
\multicolumn{2}{| c|}{\bf Error Correcting Algorithm}\\[2ex]
\hline
\hline

  &$U^0 =0$ \quad  $b^0 =0$\\[3ex]
$U^{k+1} = $ & argmin$_U \;E(U) + \frac{\lambda}{2} \| BU-b^k\|^2$ \\[2ex]
 $b^{k+1} = $ & $b^k - BU^{k+1}$\\[2ex]

 \hline
\end{tabular}    
     \label{t:modified}}

     \end{minipage}

    \caption{\footnotesize A) Bregman iteration. B) Error correcting algorithm.}
    \label{t:algorithms}
     \end{table}

To prove the above claim we take the following steps.
\begin{enumerate}
\item Show equivalence between the Bregman Iteration and the Error Correcting Algorithm.
\item Show that the sequence of Bregman iterates $\{u^k\}$ is also a minimizing sequence of $H(u)$.
\item Use item 2) to show that the solutions to the Error Correcting Algorithm converge to a solution of the constrained problem, and thus from 1) so do the Bregman iterates.
\end{enumerate}

Here again we let $X $ denote a Banach space and we consider functionals $E$ and $H$ that satisfy the following assumptions.
\begin{hypothesis}\label{h:mainfunc}
Let $E: X \rightarrow \R $ and $H: X \rightarrow \R$ be convex functionals with the property that if we look at $F(U) = E(U) + H(U)$ then $F(U)$ is coercive. That is there exist constants $1 \leq q <p$, $1 \leq r$, $\alpha_1, \beta _1>0 $ and $\alpha_2, \alpha_3\in \R$ such that $$F(U) = F(u,d) \geq \beta_1 | d|^r + \alpha_1 | \nabla u |^p + \alpha_2 |u|^q + \alpha_3.$$
\end{hypothesis}

\begin{hypothesis}\label{h:constraint}
 Let $H(U) =(\lambda/2) \|BU\|^2 $, where $B: X \rightarrow L^2$ is a bounded linear operator, define a functional satisfying $\min_{U \in X} H(U) = 0$.
 \end{hypothesis}
 
\subsection{Equivalence between algorithms}
All proofs in this subsection are based on the results from
 \cite{yin2008}.

To prove the equivalence between the two algorithms we first need this next lemma.
\begin{lemma}\label{l:welldefined}
Suppose $E$ and $H$ satisfy \cref{h:mainfunc,h:constraint}. Then, for each Bregman iteration defined using these functionals and given by the algorithm in \cref{t:algorithms} there exists a minimizer $U_k$, and subgradients $P^k, R^K$ of $\partial E(U_k)$ and $\partial H(U_k)$, respectively such that $$P^{k-1} = P^k +R^k.$$ 
\end{lemma}

\begin{proof}
Since $H(U) =  \frac{\lambda}{2} \|B U\|^2$ we note that the functional in each Bregman iteration is given by
\begin{align*}
Q_k(U) &= D^{P^{k-1}}_E(U,U^{k-1}) + H(U),\\
Q_k(U) &= E(U) - E(U^{k-1}) - \langle P^{k-1}, U-U^{k-1} \rangle + \frac{\lambda}{2} \|B U\|^2.
\end{align*}
It is not hard to check, using the definition for $F(U) = E(U)+H(U)$ and properties of the Bregman distance, that the functional $Q_k(U): X \rightarrow \R $ is convex, coercive, bounded from below, and  lower semicontinuous. Consequently each Bregman iteration $Q_k(U)$ has a minimizer $U_k$ in $X$. Moreover, the subgradient optimality condition,
\[ 0 \in  \partial Q_k(U_k) = \partial E(U_k ) - P^{k-1} + \partial H(U_k), \]
gives us
\[ P^{k-1} \in \partial E(U_k ) + \partial H(U_k), \]
showing that there is $P^k \in  \partial E(U_k )$ and $R^k \in  \partial H(U_k)$ such that $P^{k-1}= P^k + R^k$.

\end{proof}

\begin{remark} Notice that because of the relation $P^{k-1} = P^k + R^k$ we also have that $P^k = -\sum_{m=1}^k R^m$. We will use this relation in~\cref{l:uniformbound}.
\end{remark}
The following proposition establishes the equivalence between the Error Correcting Algorithm and the Bregman Iteration.
\begin{lemma}
Suppose the functionals $E$ and $H$ satisfy  \cref{h:mainfunc,h:constraint}. Then, with these functionals the two algorithms from \cref{t:algorithms} are equivalent.
\end{lemma}
\begin{proof}

To show the equivalence between the Bregman iteration, with functional $F^k_B(U)$, and the Error Correcting algorithm, with functional $F^k_{EC}(U)$, we proceed by induction. We will denote by $U$ the solutions to the Bregman iteration and by $V$ the solutions to the Error correcting algorithm. Here $P^k$ again refers to the subgradient for $E(U)$ evaluated at the minimizer $U^k$ of the functional $F^k_B$.

It is straightforward to check that for $k=1$ both algorithms reduce to finding a minimizer of the same functional,
\[ \min_U E(U) + \frac{\lambda}{2} \| BU\|^2,\]
so the base case is trivial.

In order to prove the induction step we first need to show that
 \begin{enumerate}
 \item  $B^TBU^k = B^TBV^k$, and that
 \item $P^k = \lambda B^T(b^{k-1} - BV^k)$.
 \end{enumerate}
 Notice that even for the base case, where we already know that the functionals are equivalent, it is not immediately clear that the first results holds. Indeed, if $B$ has a nontrivial kernel, the minimizer for the functional $E(U) + H(U)$ is not unique. We leave the proof of this first item to \cref{l:diff_min} where it is shown that if for any $k$ the functionals $F^k_B(U)$ and $F^k_{EC}(U)$ differ by constant, and thus the two algorithms are equivalent, then any two minimizers, $U^k$ and $V^k$, satisfy $B^TBU^k = B^TBV^k$.

Next we prove item 2). Given that $B^*BU^1 = B^*BV^1$ and recalling the for the initial iterative step, $b^0=0$, it is immediate that $P^1 = \lambda B^*(b^0 - BV^1)$. Moreover, since we know $B^*BU^k = B^*BV^k$ holds we can use induction and the definition of $P^{k}$ to prove item 2):

\[  P^k = P^{k-1} - \lambda B^*BU^k   =  \lambda B^*(b^{k-2} - BV^{k-1}) - \lambda B^*BV^k  =  \lambda B^*(b^{k-1} - BV^k).\]

We now proceed to show the equivalence of the two algorithms via induction. To that end, suppose that items 1), and 2) above hold for some $k$. Then starting with the Bregman iteration
\begin{align*}
 \min_U E(U) & - E(U^k)  - \langle P^k, U - U^k \rangle + \frac{\lambda}{2} \|BU\|^2 \\
=& \min_U E(U) - \langle P^k, U \rangle + \frac{\lambda}{2} \|BU\|^2 + C\\
=& \min_U E(U) - \lambda \langle B^*( b^{k-1} - BV^k ) ,  U \rangle + \frac{\lambda}{2} \|BU\|^2 + C\\
=&\min_U E(U) - \lambda \langle ( b^{k-1} - BV^k ) , BU \rangle + \frac{\lambda}{2}\|BU\|^2 + \frac{\lambda}{2} \|  b^{k-1} - BV^k \|^2 +  \bar{C}\\
=&\min_U E(U) + \frac{\lambda}{2} \|  ( b^{k-1} - BV^k) - BU \|^2 +  \bar{C}\\
=&\min_U E(U) + \frac{\lambda}{2} \|  b^k- BU \|^2 +  \bar{C}.
\end{align*}

Where on the third line we used 2) from the induction hypothesis, and in the last line we used the definition of $b^k$.
Since the two functionals differ by a constant the two algorithms are equivalent.

\end{proof}

\begin{lemma}\label{l:diff_min}
Suppose the functionals $E(U)$ and $H(U)= \frac{\lambda}{2}\| b- BU\|^2$ satisfy \cref{h:mainfunc}. If $U$ and $V$ are two distinct minimizers of 
\begin{equation}\label{e:energy} E(U) +  \frac{\lambda}{2} \|b - BU\|^2 +C,
\end{equation}
 where $C\in \R$, $\lambda >0 $, and $b \in L^2$, then we must have $B^TBU = B^TBV$.
\end{lemma}

\begin{proof}
Given that $U \neq V$, consider a linear combination of these two elements $Z = \alpha U +(1-\alpha)V \in X$, with $\alpha \in [0,1]$. Letting $$m = \min_{U \in X} E(U) + \frac{\lambda}{2} \|b- BU\|^2 +C,$$
we see that
\begin{align*}
E(Z) +  &\frac{\lambda}{2} \|b- BZ\|^2 +C \\
 &  \leq \alpha E(U) + (1-\alpha) E(V) +C \\
& + \frac{\lambda}{2} \left( \alpha^2 \|b-BU\|^2 + 2 \alpha(1-\alpha) \|b-BU\|\; \|b-BV\| + (1-\alpha)^2 \|b-BV\|^2 \right)\\[1.5ex]
\leq & \alpha \left ( m - \frac{\lambda}{2} \|b-BU\|^2 \right ) + ( 1- \alpha)  \left ( m - \frac{\lambda}{2} \|b-BV\|^2 \right ) \\
& +  \frac{\lambda}{2} \left( \alpha^2 \|b-BU\|^2 + 2 \alpha(1-\alpha) \|b-BU\|\; \|b-BV\| + (1-\alpha)^2 \|b-BV\|^2 \right)\\[1.5ex]
\leq & m + \frac{\lambda}{2} \left( \alpha(\alpha-1) \|b-BU\|^2 + 2 \alpha(1-\alpha) \|b-BU\|\; \|b-BV\| \right. \\
& \left. + \alpha(\alpha-1) \|b-BV\|^2 \right)\\[1.5ex]
\leq & m - \frac{\lambda}{2} \alpha ( 1- \alpha) \left( \|b-BU \| - \|b-BV\| \right )^2.
\end{align*}
This last inequality implies that $\|b-BU \| = \|b-BV\|$ and that every element in the line $Z(\alpha) = \alpha U + (1-\alpha)V $ is also a minimizer.  In particular, it follows that $ \|b-BZ(\alpha) \|^2$ is constant for all $\alpha \in [0,1]$. Therefore, the gradient of $H(U) = \|b-BU\|^2$ at $U$ in the direction of $W = V-U \neq 0$ and the gradient at $V$ in the direction of $-W$ are both zero, i.e.
\begin{align*}
\partial H(U)\mid_W =& 2\langle B^T(b-BU), W \rangle =0,\\
\partial H(V)\mid_{-W} =& 2\langle B^T(b-BV), -W \rangle =0.
\end{align*}
Subtracting these results we see that $\langle B^TBU - B^TBV, W \rangle =0$.
\end{proof}

\subsection{Properties of Bregman Iteration}
The main goal of this section is to show that the sequence of Bregman iterates,$\{ U^k\}$, generated from the algorithm in \cref{t:algorithms},  is also a minimizing sequence of $H(U)$. We state this more precisely in the following proposition.

\begin{proposition}
Suppose we have functionals $E$ and $H$ that satisfy \cref{h:mainfunc,h:constraint}. Then, the sequence $\{U^k \}$ of iterates generated by the Bregman iteration is also a minimizing sequence for $H(U)$. In particular, the sequence converges weakly to a function $\tilde{U}$ satisfying $\|B\tilde{U}\|=0$.
\label{prop:hminimizing}
\end{proposition}

We prove this proposition in a series of lemmas, which summarize the results from \cite{osher2005}. The first assertion follows from \cref{l:minimizing}, which uses the properties of the Bregman iteration stated in  \cref{l:properties}. The second assertion follows once we show that the sequence of iterates is uniformly bounded in $X$, since this implies that the sequence converges weakly to a minimizer $\tilde{U}$ of $H(U)$. In particular, to show the boundedness of the sequence:
\begin{enumerate}
\item We notice first that by \cref{h:constraint} the sum $E(U) +H(U)$ is coercive. It then follows from standard arguments and Poincar\'e's inequality that there are constants $c_1>0, c_2 \in \R$ such that the norm $\|U \|_X \leq c_1 (E(U) + H(U)+ c_2)$.
\item Then, we may conclude from \cref{l:uniformbound} that $E(U^k) +H(U^k) \leq E(\tilde{U}) $ for all $k$.
\end{enumerate}

We start with some properties of the Bregman iteration. Here we use the notation $Q_k(U)$ to represent the functional corresponding to the $k$th Bregman iteration
\[Q_k(U) = E(U) - E(U^{k-1}) - \langle P^{k-1}, U-U^{k-1} \rangle + H(U).\]
The following results follow the analysis in Osher et al \cite{osher2005}.
\begin{lemma}\label{l:properties}
Given functionals $E$ and $H$ satisfying \cref{h:mainfunc,h:constraint}, the sequence $\{U_k\} \subset X$ generated by the corresponding Bregman iteration satisfies:
\begin{enumerate}
\item Monotonicity: $H(U_k) \leq H(U_{k-1})$
\item If $E(U)< \infty$ then
\[ D^{P^k}_E( U,U^k) + D^{P^{k-1}}_E(U^k,U^{k-1})+H(U^k)-H(U)< D^{P^{k-1}}_E(U,U^{k-1}). \]
\end{enumerate}
\end{lemma}

\begin{proof}

To prove item 1) let $U^{k-1}$ and $U^k$ represent the minimizers of the $(k-1)$th and $k$th Bregman iterations, and let $P^{k-1}$ be an element in the subgradient of $E(U)$ evaluated at $U^{k-1}$. Then by applying the definition of subgradient to $P^{k-1}$ we see that,
\begin{align*}
 \langle P^{k-1} , U^k - U^{k-1} \rangle +E(U^{k-1})  \leq &   E(U^k) \\
 H(U^k) \leq  & E(U^k) - \langle P^{k-1} , U^k - U^{k-1} \rangle - E(U^{k-1})   + H(U^k)\\
 H(U^k) \leq &Q_k(U^k)  \leq Q_k(U^{k-1}) = H(U^{k-1}).
 \end{align*}
 
Where the second inequality holds because $U^k$ minimizes $Q_k(U^k)$.

To prove item 2) we use the definition of the Bregman distance to simplify the following expression
\begin{align*}
 D^{P^k}_E(U, U^k) -& D^{P^{k-1}}_E(U,U^{k-1}) + D^{P^{k-1}}_E(U^k,U^{k-1})  \\
 =& \; E(U) -E(U^k) - \langle P^k, U-U^k \rangle  + E(U^{k-1}) - E(U) \\
 &+ \langle P^{k-1}, U-U^{k-1} \rangle +  E(U^k) -E(U^{k-1}) - \langle P^{k-1}, U^k-U^{k-1} \rangle\\
  = & - \langle P^k, U-U^k \rangle + \langle P^{k-1}, U-U^{k-1} \rangle - \langle P^{k-1}, U^k-U^{k-1} \rangle\\
  =&  \langle P^{k-1}- P^k, U-U^k \rangle.
\end{align*}
From \cref{l:welldefined} we know that $P^{k-1} = P^k +R^k$ , with  $R^k \in \partial H(U^k)$. This allows us to simplify the expression further leading to
\[  D^{P^k}_E(U, U^k) - D^{P^{k-1}}_E(U,U^{k-1}) + D^{P^{k-1}}_E(U^k,U^{k-1})    =  \langle R^k, U-U^k \rangle \leq H(U) - H(U^k). \]
After a rearrangement this gives the desired result,
\[ D^{P^k}_E( U,U^k) + D^{P^{k-1}}_E(U^k,U^{k-1})+H(U^k)-H(U)< D^{P^{k-1}}_E(U,U^{k-1}). \]

\end{proof}

 This next proposition implies that the sequence of Bregman iterates $\{U^k\}$ is a minimizing sequence for $H(U)$.

\begin{lemma}\label{l:minimizing}
Suppose $E(U)$ and $H(U)$ satisfy \cref{h:mainfunc,h:constraint} and that $\tilde{U}$ is a minimizer of $H(U)$, with $E(\tilde{U})<\infty$. Then, the sequence $\{U_k\} \subset X$ generated by the Bregman iteration in \cref{t:algorithms} satisfies $$H(U^k) \leq H(\tilde{U}) + \frac{E(\tilde{U})}{k}.$$
\end{lemma}

\begin{proof}
The result follows from adding item 2) in \cref{l:properties} for integers 1 through $k$: 
\begin{equation}\label{e:niceineq}
 D^{P^k}_E(\tilde{U},U^k) + \sum_{m=1}^k \left[ D^{P^{m-1}}_E(U^m,U^{m-1}) + H(U^m) - H(\tilde{U}) \right] \leq D^0(\tilde{U},U^0).
 \end{equation}
Using the monotonicity property, i.e. $H(U^m)\leq H(U^{m-1})$, we can replace $H(U^m)$ with $H(U^k)$ for all $m=1,2,\cdots, k$. In addition because $D^{P^{m-1}}_E(U^m, U^{m-1}) \geq 0$ the above inequality can be simplified to
\[ D^{P^k}_E(\tilde{U},U^k) +  k \left [H(U^k) - H(\tilde{U}) \right] \leq  D^0(\tilde{U},U^0) =E(\tilde{U}). \]
Lastly, because the Bregman distance is always nonnegative we can rearrange the terms in this last inequality to obtain the desired result $$H(U^k) \leq H(\tilde{U}) + E(\tilde{U})/k.$$
\end{proof}

\begin{Remark}\label{r:properties}
From the inequality \cref{e:niceineq} one also obtains the following properties for the sequence of Bregman iterates:
\begin{enumerate}
\setlength \itemsep{2ex}
\item $ \sum_{m=1}^k D^{P^{m-1}}_E(U^m,U^{m-1}) \leq E(\tilde{U} )$.\\

 Since in addition the $\min H(U) =0$ over $X$, we also have that 
 
\item $ \sum_{m=1}^k H(U^m) \leq E(\tilde{U}) $ as well as
\item $kH(U^k) \leq E(\tilde{U})$.
\end{enumerate}
\end{Remark}

In this next lemma we show that if the functionals $E$ and $H$ satisfy the above hypothesis and $\{U^k\}$ is a minimizing sequence, then sequence of values  $a_k = E(U^k) + H(U^k)$ is uniformly bounded . Since $\|U\|_X \leq c_1( E(U)+H(U) +c_2)$ for some constants $c_1>0, c_2 \in \R$, it follows that the minimizing sequence $\{U^k\}$ is uniformly bounded and therefore converges weakly to an element in $X$.

\begin{lemma}\label{l:uniformbound}
Suppose $E(U)$ and $H(U)$ satisfy \cref{h:mainfunc,h:constraint} and that $\tilde{U}$ is a minimizer of $H(U)$, with $E(\tilde{U})<\infty$. Then, the sequence $\{U_k\} \subset X$ generated by the Bregman iteration in \cref{t:algorithms} satisfies $$E(U^k) + H(U^k) \leq C E(\tilde{U}).$$
\end{lemma}
\begin{proof}
To show the result we use item 1) from \cref{r:properties}
\begin{align*}
E(\tilde{U}) \geq & \sum_{m=1}^k D^{P^{m-1}}_E(U^m,U^{m-1})\\
\geq &  \sum_{m=1}^k  \left( E(U^m) - E(U^{m-1}) - \langle P^{m-1}, U^m-U^{m-1} \rangle  \right) \\
\geq & E(U^k) - E(U^0) -  \sum_{m=1}^k  \langle P^{m-1}, U^m-U^{m-1} \rangle\\
\geq & E(U^k) - E(U^0) - \left(  \sum_{m=1}^k  \langle P^{m-1}, U^m-\tilde{U} \rangle - \langle P^{m-1}, U^{m-1}-\tilde{U} \rangle \right)   \\
\geq & E(U^k) - E(U^0) - \langle P^{k-1}, U^k - \tilde{U} \rangle +   \sum_{m=1}^{k-1}  \langle P^m - P^{m-1}, U^m-\tilde{U} \rangle.
\end{align*}
Using the results from \cref{l:welldefined}, $P^{m-1} = P^m + R^m$ and $P^k = - \sum_{m=1}^k R^m$ we can write
\begin{align*}
E(\tilde{U}) \geq & E(U^k) - E(U^0) +  \sum_{m=1}^{k-1}  \langle R^m, U^k - \tilde{U} \rangle -  \sum_{m=1}^{k-1} \langle R^m, U^m - \tilde{U} \rangle\\
 \geq & E(U^k) - E(U^0) +  \sum_{m=1}^{k-1}  \langle R^m, U^k \rangle -  \sum_{m=1}^{k-1} \langle R^m, U^m  \rangle.
\end{align*}
\end{proof}
Since $R^m \in \partial H(U^m) = \lambda B^*BU^m$ we have 
\begin{align*}
E(\tilde{U}) \geq & E(U^k) - E(U^0)+  \lambda  \sum_{m=1}^{k-1} \langle BU^m, BU^k \rangle - \lambda \sum_{m=1}^{k-1} \| BU^m \|^2\\
 \geq & E(U^k) - E(U^0) -  \frac{\lambda}{2}  \sum_{m=1}^{k-1} \left (\|BU^m\|^2 + \|BU^k\|^2 \right ) - \lambda \sum_{m=1}^{k-1} \|BU^m\|^2\\
 \geq & E(U^k) +H(U^k) - E(U^0) - k H(U^k) -3  \sum_{m=1}^{k-1} H(U^m).
\end{align*}

Since $\min_{U \in X} H(U) =0$, we can use \cref{r:properties} to obtain
\[ E(\tilde{U}) \geq  E(U^k) +H(U^k) - E(U^0) - 4E(\tilde{U}),\]
which yields the result of the lemma
\[ E(U^k) +H(U^k) \leq 5 E(\tilde{U}). \]

\subsection{ Convergence to solution of constrained problem}

We have shown that the sequence $\{ U_k\}$ of Bregman iterates is a minimizing sequence for $H(U) = \frac{\lambda}{2} \|BU\|^2$. In particular this implies that the sequence converges weakly to a function $U^* \in X$ with the property that $\|BU^* \| =0$. Because the Bregman iteration and the Error correcting algorithm are equivalent we also have that $U^*$ is a solution to an iterate of the latter. In this next proposition we further show that if $U^*$ is a solution to the Error Correcting algorithm which satisfies $\| BU^*\| =0$, then it must also be a solution to the original constrained problem
\begin{align}\label{e:constrained} 
&\min_{U \in X} E(U) \quad \mbox{subject to} \quad \|BU \| =0,\\ \nonumber
&\min_{(u,d) \in X} \overline{W}_1(d) + \tilde{V}(x,u)  \quad \mbox{subject to} \quad \|\partial_x u - d \| =0.
\end{align}
The proof we present here follows the analysis in \cite{goldstein2009}.
\begin{proposition}
Suppose the functionals $E(U)$ and $H(U)$ satisfy \cref{h:mainfunc,h:constraint}. Consider the Error Correcting algorithm stated in \cref{t:algorithms} and suppose an iterate $U^*$ satisfies $\| BU^*\| =0$. Then $U^*$ is a solution to the original constrained problem \cref{e:constrained}.
\label{prop:fixed}
\end{proposition}

\begin{proof}
Since $U^*$ 
is a fixed point for the Error Correcting algorithm there is a $b^*$ such that $$U^* = \mbox{ argmin}_{U \in X} E(U) + \frac{\lambda}{2} \| BU -b^*\|.$$
Suppose now that $\bar{U}$ is a solution to the original constrained problem \cref{e:constrained}, then $\|B\bar{U} \| =0$. Because $U^*$ also satisfies the same constrain, we obtain the following relation $\|BU^* -b^* \| = \| B\bar{U} - b^*\|$. We can now use this to show that $U^*$ is a solution to  \cref{e:constrained}. Indeed because $U^*$ is a minimizer of the Error Correcting functional we see that
 \begin{align*}
E(U^*) + \frac{\lambda}{2} \|BU^* -b^*\| &\leq E(\bar{U}) + \frac{\lambda}{2} \|B\bar{U} -b^*\|\\
E(U^*) &\leq E(\bar{U}).
\end{align*}
The last inequality shows that $U^*$ is also a minimizer for $E(U)$ and thus solves \cref{e:constrained}.

\end{proof}

\section*{Acknowledgments}
GJ acknowledges the support from the National Science Foundation through grants DMS-1503115 and DMS-1911742. SV was partially supported by the Simons Foundation through awards 524875 and 560103 and also partially supported by the NSF through award DMR-1923922. Portions of this work were carried out when SV was visiting the Center for Nonlinear Analysis at Carnegie Mellon University and the Oxford Center for Industrial and Applied Math. 

\bibliographystyle{siamplain}
\bibliography{ModBregman}
\end{document}